\documentclass[12pt]{amsart}
\usepackage[active]{srcltx}
\usepackage{a4wide}
\usepackage{amsthm,amsfonts,amsmath,mathrsfs,amssymb}
\usepackage{dsfont}
\usepackage[T1]{fontenc}
\usepackage[utf8]{inputenc}
\usepackage{fourier}
\usepackage[dvipsnames,usenames]{color}

\newcommand{\D}{\mathbb{D}}

\newcommand{\RR}{\mathbb{R}}
\newcommand{\C}{\mathbb{C}}
\newcommand{\Z}{\mathbb{Z}}
\newcommand{\T}{\mathbb{T}}

\newcommand{\Cp}{C_{\varphi}}

\newcommand{\GH}{\mathscr{G}}
\newcommand{\Hi}{{\mathscr{H}}^\infty}
\newcommand{\Ho}{{\mathscr{H}}^1}
\newcommand{\Ht}{{\mathscr{H}}^2}

\newcommand{\Hp}{{\mathscr{H}}^p}
\newcommand{\Ha}{{\mathscr{H}}^{2^{\alpha+1}}}
\newcommand{\Da}{{\mathscr{D}}_\alpha}

\newcommand{\Real}{\operatorname{Re}}
\newcommand{\Imag}{\operatorname{Im}}

\newtheorem{theorem}{Theorem}[section]

\newtheorem{lemma}{Lemma}[section]

\begin{document}

\title[Approximation numbers of composition operators]{Approximation numbers of composition operators \\ on $H^p$ spaces of Dirichlet series}



\author[F. Bayart]{Fr\'ed\'eric Bayart}
\address{Clermont Universit\'e, Universit\'e Blaise Pascal, Laboratoire de Math\'ematiques, BP 10448, F-63000 CLERMONT-FERRAND -
CNRS, UMR 6620, Laboratoire de Math\'ematiques, F-63177 AUBIERE} \email{bayart@math.univ-bpclermont.fr}

\author[H. Queff\'{e}lec]{Herv\'{e} Queff\'{e}lec}
\address{Universit\'{e} Lille Nord de France, USTL, Laboratoire Paul Painlev\'{e} UMR. CNRS 8524, F--59 655 Villeneuve d'Ascq
Cedex, France}\email{herve.queffelec@univ-lille1.fr}


\author[K. Seip]{Kristian Seip}
\address{Department of Mathematical Sciences, Norwegian University of Science and Technology,
NO-7491 Trondheim, Norway} \email{seip@math.ntnu.no}
\thanks{The second author is supported by the Research Council of Norway grant 227768. This paper was initiated during the research program \emph{Operator Related Function Theory and Time-Frequency Analysis} at the Centre for Advanced Study at the Norwegian Academy of Science and Letters in Oslo during 2012--2013.}

\begin{abstract}
By a theorem of the first named author, $\varphi$ generates a bounded composition operator on the Hardy space $\Hp$of Dirichlet series
($1\le p<\infty$) only if  $\varphi(s)=c_0 s+\psi(s)$, where $c_0$ is a nonnegative integer and $\psi$ a Dirichlet series with the following mapping properties: $\psi$ maps the right half-plane into the half-plane $\Real s >1/2$ if $c_0=0$ and is either identically zero or maps the right half-plane into itself if $c_0$ is positive. 
It is shown that the $n$th approximation numbers of bounded composition operators on $\Hp$ are bounded below by a constant times $r^n$ for some $0<r<1$ when $c_0=0$ and bounded below by a constant times $n^{-A}$ for some $A>0$ when $c_0$ is positive. Both results are best possible. Estimates rely on a combination of soft tools from Banach space theory ($s$-numbers, type and cotype of Banach spaces, Weyl inequalities, and Schauder bases) and a certain interpolation method for $\Ht$, developed in an earlier paper,  using estimates of solutions of the $\overline{\partial}$ equation. A transference principle from $H^p$ of the unit disc is discussed, leading to explicit examples of compact composition operators on $\Ho$ with approximation numbers decaying at a variety of sub-exponential rates. Finally, a new Littlewood--Paley formula is established, yielding a sufficient condition for a composition operator on $\Hp$ to be compact.
\end{abstract}

\maketitle

\section{Introduction and statement of main results}

In the recent work \cite{QS1}, we studied the rate of decay of the approximation numbers of compact composition operators on the Hilbert space $\Ht$, which consists of all ordinary Dirichlet series $f(s)=\sum_{n=1}^{\infty}b_n n^{-s}$ such that
\[ \| f\|^2_{\Ht} := \sum_{n=1}^\infty \vert b_n\vert^2 <\infty. \]
The general motivation for undertaking such a study is that the decay of the approximation numbers is a quantitative way of studying the compactness of a given operator, yielding more precise information than what for instance its membership in a Schatten class does.
The purpose of the present paper is to take the natural next step of making a similar investigation in the case when $\Ht$ is replaced by the Banach spaces $\Hp$ for $1\le p < \infty$; here we follow \cite{BAY} and define ${\Hp}$ as the completion of the set of Dirichlet polynomials $P(s)=\sum_{n=1}^N b_n n^{-s}$ with respect to the norm 
\[ \Vert P\Vert_{{\Hp}}=\left(\lim_{T\to \infty} \frac{1}{T}\int_{0}^T \vert P(it)\vert^{p}dt\right)^{1/p}. \]
We consider this a particularly interesting case because operator theory on these spaces so far is poorly understood and appears intractable by standard methods.  

Our starting point is the first named author's work on $\Hp$ and boundedness and compactness of the composition operators acting on these spaces \cite{BAY, BAYA}. A basic fact proved in \cite{BAY} is that $\Hp$ consists of functions analytic in the half-plane $\sigma:=\Real s>1/2$. This means that $C_{\varphi}f:= f\circ \varphi$ defines an analytic function whenever $f$ is in $\Hp$ and $\varphi$ maps this half-plane into itself. But more is clearly needed for $C_{\varphi}$ to map $\Hp$ into $\Hp$. In particular, we need to consider other half-planes as well and introduce therefore again the notation
\[\C_\theta:=\{s=\sigma+i t:\ \sigma>\theta\},\]
where $\theta$ can be any real number. Following the work of Gordon and Hedenmalm \cite{GORHED}, we say that an analytic function $\varphi$ 
on $\C_{1/2}$ belongs to the Gordon--Hedenmalm class $\GH$ if it can be represented as
\[ \varphi(s)=c_{0}s+\sum_{n=1}^\infty c_n\, n^{-s}=:c_{0}s+\psi(s), \]
where $c_0$ is a nonnegative integer and $\psi$ is a Dirichlet series that is uniformly convergent in each half-plane\ \  $\C_\varepsilon$ ($\varepsilon>0$) and is either identically $0$ or has the mapping properties $\psi(\C_0)\subset \C_0$ if $c_0\geq 1$ and $\psi(\C_0)\subset \C_{1/2}$ if $c_0=0$. This terminology  is justified by the result from \cite{GORHED} saying that $C_\varphi$ is bounded on $\Ht$ if and only if $\varphi$ belongs to $\GH$. (See \cite[Theorem~1.1]{QS1} for this particular formulation of the result.) The $\Hp$ version of the Gordon--Hedenmalm theorem reads as follows \cite{BAYA}. 

\begin{theorem}\label{FRED} Assume that $\varphi:\C_{1/2}\to \C_{1/2}$ is an analytic map and that $1\leq p<\infty$.  \begin{enumerate}
\item[(a)] If $C_\varphi$ is bounded on ${\Hp}$,  then $\varphi$ belongs to $\GH$. 

\item[(b)] $C_\varphi$ is a contraction on $\Hp$ if and only if $\varphi$ belongs to $\GH$ and $c_0\geq 1$. 

\item[(c)] If $\varphi$ belongs to $\GH$ and $c_0=0$, then $C_\varphi$ is bounded on ${\Hp}$ whenever $p$ is an even integer.

\end{enumerate}
\end{theorem} 

The curious fact that part (c) of this theorem only covers the case when $p$ is an even integer can be directly attributed to an interesting feature of $\Hp$. To see this, we recall that  $\Hp$ can be identified  isometrically with $H^{p}(\T^\infty)$, which is the $H^p$ space of the infinite-dimensional polydisc $\T^\infty$, via the so-called Bohr lift. This space is a subspace of the Lebesgue space $L^{p}(\T^\infty)$ with respect to normalized Haar measure on $\T^{\infty}$. The main difference with classical $H^p$ spaces is that, even for $1<p<\infty$, $p\neq 2$, ${\Hp} =H^{p}(\T^\infty)$ is not complemented in $L^{p}(\T^\infty)$ \cite{EBE}. As a consequence, there seems to be little hope to obtain a useful description of the dual space of $\Hp$. This is a serious obstacle and makes it hard to employ familiar techniques such as interpolation in the Riesz--Thorin or Lions--Peetre sense. We refer to \cite{SS, OLS} for further details about the anomaly of $\Hp$. There are additional obstacles as well, since familiar Hilbert space techniques such as orthogonal projections, frames and Riesz sequences, 
 and equality of various $s$-numbers (like approximation, Bernstein, Gelfand  numbers) are no longer available.


We have found ways to circumvent these difficulties to obtain results that, at least partially, parallel those from \cite{QS1}. To state our first result, we recall that the $n$th approximation number $a_n(T)$ of a bounded operator on a Banach space $X$ is the distance in the operator norm from $T$ to operators of rank $<n$. One of our main theorems is a direct analogue and indeed an improvement of \cite[Theorem 1.1]{QS1}, showing again the crucial dependence on the parameter $c_0$: 

 \begin{theorem}\label{General} Assume  that $c_0$ is a nonnegative integer, $p\ge 1$, and that $\varphi(s)=c_0 s+\sum_{n=1}^\infty c_n n^{-s}$ is a nonconstant function that generates a bounded composition operator $C_{\varphi}$ on ${\Hp}$. 
\begin{itemize}
\item[(a)] If $c_0=0$, then $a_n(C_{\varphi})\gg \delta^n$ for some\ $0<\delta<1$.
\item [(b)]If $c_0\geq1$, then  $a_n(C_{\varphi})\geq \delta_p (n\,\log n)^{-\Real c_1}$, where $\delta_{p}>0$ only depends on $p$.
 In particular,\  if  $\ \Real c_1>0$, then 
\begin{equation}\label{bill}\sum_{n=1}^\infty [a_{n}(C_\varphi)]^{1/\Real c_1}=\infty.\end{equation} 

\end{itemize}
These lower bounds  are optimal.
\end{theorem}

As in \cite{QS1}, the notation $f(n)\ll g(n)$ or equivalently $g(n)\gg f(n)$ means that there is a constant $C$ such that $f(n)\le C g(n)$ for all $n$ in question.

The proof of Theorem~\ref{General} follows the same pattern as the proof of \cite[Theorem 1.1]{QS1}, with an additional ingredient allowing a sharper estimate (new even for $p=2$) in the case $c_0\geq 2$. The general strategy is based on fairly soft functional analysis, which turns out to remain valid in the context of our spaces $\Hp$. It leads to certain general estimates for approximation numbers of $C_\varphi$, which are made effective thanks to a hard technical device involving $\overline{\partial}$ correction arguments. It is perhaps surprising to find such methodology in this context, but it proved to be quite efficient in the case $p=2$ \cite{QS1}. Unexpectedly, this auxiliary Hilbert space result can be used again in the present setting without any essential changes.

Contrary to what happens for composition operators on $H^p(\mathbb D)$, continuity or compactness of composition operators on $\mathcal H^p$ can not be immediately inferred from what is known when  $p=2$. Indeed, no nontrivial
sufficient condition for a composition operator to be compact on $\Hp$, $p\neq 2$, has been found in previous studies. We have therefore chosen to include in the present paper a sufficient condition when $c_0\geq 1$, similar to the condition given in \cite{BAYA} for composition operators on $\Ht$. As in that paper, our condition will depend on a Littlewood--Paley formula, but the proof of \cite{BAYA} can not be easily adapted since it depends crucially on the Hilbert space structure of $\mathcal H^2$. To state our condition, we need to recall that
 the Nevanlinna counting function of a function $\varphi$ in $\mathcal G$ is defined by
$$\mathcal N_ \varphi(s)=\left\{
\begin{array}{ll}
\sum_{w\in\varphi^{-1}(s)}\Real w&\textrm{if }s\in\varphi(\mathbb C_0)\\
0&\textrm{otherwise}.
\end{array}\right.
$$

\begin{theorem}\label{sufficientcompact} Assume  that $\varphi(s)=c_0  s+\psi(s)$ is in $\mathcal G$ with $c_0\geq 1$ and that
\begin{itemize}
\item[(a)] $\Imag \psi$ is bounded on $\C_0$;
\item[(b)] $\mathcal N_\varphi(s)=o\big(\Real s\big)$ when $\Real s\to 0^+$.
\end{itemize}
Then $C_\varphi$ is compact on $\Hp$ for $p\ge 1$.
\end{theorem}
Note that in the special case when $\varphi$ is univalent, we have either $\mathcal N_\varphi(s)=\Real \varphi^{-1}(s)$ or $\mathcal N_\varphi(s)=0$. This means that assumption (b) can be rephrased as saying that $\Real \psi(s)/ \Real s \to +\infty$ when $\Real s \to 0^+$.

In the final part of the paper, we will discuss a transference principle that was established in \cite{QS1}, showing that symbols of composition operators on $H^{p}(\D)$, via left and right composition with two fixed analytic maps, give rise to composition operators on $\Hp$. The idea is to use this principle to transfer estimates for approximation numbers from $H^p(\D)$. This leads to satisfactory results when $p=1$, but, surprisingly, for no other values of $p\neq 2$.  As an example, we mention that it allows us to prove the following result.  

 \begin{theorem}\label{example} There exists a function $\varphi(s)=\sum_{n=1}^\infty c_n n^{-s}$ in $\GH$ such that $C_\varphi$ is bounded on $\Ho$ with approximation numbers verifying 
$$a\,e^{-b\sqrt n}\leq a_{n}(C_\varphi)\leq a'\,e^{-b'\sqrt n}.$$
\end{theorem}

This theorem follows, via our transference principle, from estimates from \cite{ROQULI} for composition operators on $H^p(\T)$ associated with lens maps on $\D$. More general statements about the range of possible decay rates for approximation numbers of $C_\varphi$ on $\Ho$ can be worked out, using our transference principle and the methods of our recent work \cite{QS2}. However, at present, we are not able to reach the same level of precision for approximation numbers of slow decay. This is related to a certain local embedding inequality, known to be valid only when $p$ is an even integer; this is also the reason for the constraint in part (c) of Theorem \ref{FRED}. As to the success of our method when $p=1$, the point is that we are able to resort to estimates for interpolating sequences for $\Ht$, in contrast to what can be done for general $p>1$. Indeed, essentially nothing is known about the interpolating sequences for $\Hp$ when $1<p<\infty, p\neq 2$.


Our study requires a fair amount of background material and function and operator theoretic results pertaining to $\Hp$. More specifically, the following list shows what will be covered in the remaining part of the paper: \begin{itemize}
\item Section 2 is devoted to some preliminaries and definitions on operators and Banach spaces.
\item In Section 3, we recall some general properties of $\Hp$: Schauder basis, Bohr lift, etc\ldots
\item In Section 4,  we have collected a number of basic functional analytic results as well as more specific results about Carleson measures and interpolating sequences, including estimates relying on our previous work \cite{QS1}.
\item Section 5 is devoted to a new Littlewood--Paley type formula for $\Hp, p\neq 2$, which is subsequently used to prove Theorem~\ref{sufficientcompact}. 
\item Section 6 establishes general lower bounds for $a_{n}(C_\varphi)$ using norms of Carleson measures and constants of interpolation.
\item Section 7 gives several proofs of Theorem \ref{General}.
\item Section 8 shows the optimality of the previous bounds.
\item The final Section 9 is devoted to our transference principle and to the proof of Theorem \ref{example}.
We also discuss two basic problems (the local embedding inequality and interpolating sequences for $\Hp$) that hinder further progress in our particular context as well as in our general understanding of $\Hp$.
\end{itemize}

Throughout the paper, we insist on exhibiting methods, and sometimes we give several proofs of the same result.

 \section{Preliminaries}
\subsection{$s$-numbers}\label{sn}
 Let $T:X\to X$ be a bounded operator from a Banach space into itself, and let $(\lambda_{n}(T))_{n\geq 1}$ be the sequence of its eigenvalues, arranged in descending order.  We attach to this operator three sequences 
of so-called $s$-numbers, which dominate in a vague sense the sequence   $(\lambda_{n}(T))_{n\geq 1}$ and whose decay is designed to evaluate the degree of compactness of $T$ in a quantitative way:
\begin{itemize}
\item[(a)] Approximation numbers 
$$a_{n}(T)=\inf\{\Vert T-R\Vert : \hbox{\ rank}\ R<n\}$$
\item[(b)] Bernstein numbers 
$$b_{n}(T)=\sup_{\dim E=n}\Big[\inf_{x\in S_E}\Vert Tx\Vert\Big]$$
\item[(c)] Gelfand numbers 
$$c_{n}(T)=\inf\{\Vert T|_E\Vert : \hbox{\ codim}\ E<n\}.$$
\end{itemize}
 Note that the first and third sequences are defined as min--max and the  second  as a max--min.  We have
 \[a_{n}(T)=b_{n}(T)=c_{n}(T)=a_{n}(T^{\ast})=\lambda_{n}(|T|)\]
if $X$ is a Hilbert space and $|T|$ denotes $\sqrt{T^{\ast}T}$ and 
\begin{equation}\label{stefan} a_{n}(T)\geq \max (b_{n}(T),c_{n}(T)) \quad \text{and}\quad a_{n}(T)\geq a_{n}(T^{\ast})\end{equation}
if $X$ is a Banach space, where the latter inequality is an equality when $T$ is compact. Moreover, 
\begin{equation}\label{stefane} a_{n}(T)\leq 2\sqrt{n}\, c_{n}(T) ,\end{equation}
which follows from the fact that any $n$-dimensional subspace $E$ of a Banach space $X$ is complemented in $X$ through a projection of norm at most $\sqrt n$ \cite[p. 114]{PIE}.

Finally, we will combine a basic property of composition operators with a matching property of approximation numbers:
\begin{itemize}
\item[(i)] (Non abelian semi-group property)

$$C_{\varphi_1\circ \varphi_2}=C_{\varphi_2}\circ C_{\varphi_1}$$
\item[ (ii)] (Ideal property)
$$a_{n}(ATB)\leq \Vert A\Vert a_{n}(T)\Vert B\Vert.$$
\end{itemize}
For detailed information on $s$-numbers, we refer to the articles  \cite{PIE,PIET} or to the books  \cite{CAST, PIETS}. 

\subsection{The Weyl inequalities of Johnson--K\"{o}nig--Maurey--Retherford and Pietsch}
We borrow  the following theorem from the famous paper \cite{JKMR}, which extended for the first time, under an additive form, the Weyl inequalities for Hilbert spaces to a Banach space  setting. 
\begin{theorem} [Johnson--K\"onig--Maurey--Retherford] \label{four}Let $T:X\to X$ be a bounded and power-compact linear operator from a Banach space $X$ to itself and $\ 0<r<\infty$. Then
\begin{equation}\label{bernard} \Vert (\lambda_{j}(T))\Vert_{\ell^r}\leq c_{r} \Vert (a_{j}(T))\Vert_{\ell^r}.\end{equation}
\end{theorem}
Soon after this result was established, Pietsch found the following multiplicative improvement (see \cite[p. 156]{PIE} or \cite[Lemma 13]{PIET}) which we will use as well. 
\begin{theorem} [Pietsch] \label{multi}Let $T:X\to X$ be a bounded and power-compact linear operator from a Banach space $X$ to itself. Then
\begin{equation}\label{albrecht} \vert \lambda_{2n}(T)\vert\leq e\,\Big(\prod_{j=1}^{n} a_{j}(T)\Big)^{1/n}.\end{equation}
\end{theorem}

\section{ General properties of  $\Hp$}

\subsection{A Schauder basis for $\Hp$ and the partial sum operator}

We will make use of the following  result, first found by Helson \cite{HEL} (see also \cite[p. 220]{RUD}) in the framework of ordered groups and then reproved in a more concrete way in the context of ${\Hp}$-spaces by Aleman, Olsen, and Saksman \cite{ALOLSA}.
\begin{theorem}[Helson--Aleman--Olsen--Saksman] \label{thmhaos} For $1<p<\infty$, the sequence $(e_n)=(n^{-s})$  is a (conditional) Schauder basis for ${\Hp}$. \end{theorem}
For every positive integer $N$, we define the partial sum operator $S_N:\Hp\to \Hp$ by the relation
\[ S_N\left(\sum_{n=1}^\infty x_n e_n\right):= \sum_{n\leq N} x_n\,e_n. \]  In \cite{ALOLSA}, Theorem~\ref{thmhaos} was obtained as a consequence of the uniform boundedness of $S_N$, i.e., the fact that there exists a constant $C=C_p$ such that 
\[ \Vert S_N f \Vert_{\Hp} \leq C\Vert f\Vert_{\Hp} \] 
for every $f$ in $\Hp$. It is a general functional analytic fact that a complete sequence is a Schauder basis for a Banach space $X$ if and only if the associated partial sum operators $S_N$ are uniformly bounded  \cite{LIQU}. This leads to the following result.
\begin{lemma}[Contraction principle for Schauder bases] \label{contract} Let $X$ be a Banach space and assume that $(e_n)_{n\ge 1}$ is a Schauder basis for $X$. Let $N$ be a positive integer and $(\lambda_n)_{n\geq N}$ a nonincreasing sequence of nonnegative numbers. Then for every  $f=\sum_{n=1}^\infty x_n\, e_n$ in $X$,  
\begin{equation}\label{prinzip}\left\Vert \sum_{n\geq N}\lambda_n\,x_n\, e_n\right\Vert\leq 2C\lambda_N \Vert f\Vert,\end{equation}
where $C=\sup_n \| S_n\|$.
\end{lemma}
\begin{proof} We write $x_n e_n=S_nf-S_{n-1}f$ so that we get 
\[ \sum_{n\geq N}\lambda_n\,x_n\, e_n =- \lambda_{N} S_{N-1}f+\sum_{n\ge N} (\lambda_n-\lambda_{n+1})S_n f.\]
\end{proof}

The sequence $(n^{-s})$ fails to be a basis for $\mathscr{H}^1$, but not by much, as expressed by the following theorem.
\begin{theorem}\label{parsum}There exists a constant $C$ such that \begin{equation}\label{helsak}\Vert S_{N}f\Vert_{\mathscr{H}^1}\leq C\, \log N\, \Vert f\Vert_{\mathscr{H}^1}\end{equation}
for every $f$ in $\Ho$.

\end{theorem}
\begin{proof} We can appeal  to Helson's work  \cite{HEL} which deals with a compact connected group $G$ and its ordered dual  $\Gamma$. We take $G=\T^\infty,\ \Gamma=\Z^{(\infty)}$ along with the order 
$$\alpha\leq \beta \hbox{\quad if}\quad \sum_{j\geq 1} \alpha_j\log p_j\leq \sum_{j\geq 1} \beta_j\log p_j.$$
Here $\alpha=(\alpha_j),\ \beta=(\beta_j)$, and $(p_j)$ denotes the sequences of primes. 

Given $0<p<1$, Helson's result implies fairly easily that
\[ \Vert S_{N}f\Vert_{\Hp}^{p}\ll \left(\cos \pi p/2\right)^{-1} \Vert f\Vert_{\Ho}^{p}\ll (1-p)^{-1} \Vert f\Vert_{\Ho}^{p}\] 
when $f$ is in $\Hp$. Let now $f(s)=\sum_{n=1}^\infty b_{n}n^{-s}$ be in $\Ho$.  Observe that $|b_n|\leq \Vert f\Vert_{\Ho}$ which implies the pointwise estimate $|S_{N}f|\leq N\,\Vert f\Vert_{\Ho}$. Using the Bohr lift and integrating over $\T^\infty$ with its Haar mesure $m_\infty$, we obtain 

\begin{eqnarray*}
\Vert S_{N}f\Vert_{\Ho}&= & \int|S_{N}f|dm_\infty=\int|S_{N}f|^{1-p}|S_{N}f|^{p}dm_\infty \\ 
& \le & N^{1-p}(\Vert f\Vert_{\Ho})^{1-p}\int|S_{N}f|^{p}dm_\infty \\ 
& \ll & N^{1-p}(\Vert f\Vert_{\Ho})^{1-p}(1-p)^{-1}(\Vert f\Vert_{\Ho})^{p}= N^{1-p}(1-p)^{-1}\Vert f\Vert_{\Ho}.\end{eqnarray*} 
Now choosing $p=1-1/\log N$, we arrive at \eqref{helsak} since then $N^{1-p}=e$. 
\end{proof} 
We will now give an alternate and self-contained proof of Theorem~\ref{parsum}. This proof is an application of a simple but powerful identity shown by Eero Saksman to the second-named author \cite{EERO}, which most likely will have other applications in the study of the spaces $\Hp$.  The initial application Saksman had in mind  concerned  the conjugate exponent $p=\infty$ or, in other words,  the Banach algebra $\Hi$ which consists of Dirichlet series that define bounded analytic functions on $\C_0$.  Equipped with the natural $H^\infty$ norm on $\C_0$, ${\Hi}$ is isometrically  equal to the multiplier algebra of ${\Hp}$ \cite{BAY, HLS}.

We will use the notation 
\[\widehat{\psi}(\xi)=\int_{-\infty}^\infty e^{-it\xi} \psi(t)dt\] for the Fourier transform of a function $\psi$ in $L^1(\RR)$. If $\psi$ is in $L^1(\RR)$ and $f(s)=\sum_{n=1}^\infty b_n n^{-s}$ is a Dirichlet series that is absolutely convergent in some half-plane $\C_\theta$, then the function 
\begin{equation}\label{sakiden} P_{\psi}f(s):=\sum_{n=1}^\infty b_n n^{-s}\widehat{\psi}(\log n)=\int_{-\infty}^\infty f(s+it)\psi(t)dt, \quad s\in \C_\theta, \end{equation}
is again a Dirichlet series. It is clear that $P_\psi f$ is absolutely convergent wherever $f$ is absolutely convergent. If $\psi$ is in $L^{1}(R)$ with $\widehat{\psi}$ compactly supported and $f(s)=\sum_{n=1}^\infty b_n n^{-s}$ is a Dirichlet series that is bounded in some half-plane, then we have the identity between $P_{\psi} f(s)$ and the integral on the right-hand side of \eqref{sakiden}
throughout that half-plane. If $f$ is in $\Hp$ for some $p$, then we may write
\begin{equation}\label{saksident} P_{\psi}f=\int_{-\infty}^\infty (T_{t}f)\,\psi(t)dt,\end{equation}
where the right-hand side denotes a vector-valued integral in $\Hp$ and $T_{t}:\Hp\to \Hp$ is the vertical translation-operator defined by $T_{t}f(s)=f(s+it)$. The utility of \eqref{sakiden} rests on the fact that $T_{t}$ acts isometrically on $\Hp$ for every $1\le p \le \infty$. Saksman's vertical convolution formula \eqref{sakiden} can now be applied to yield the following result.

\begin{lemma} If $ \psi$ is in $ L^{1}(\mathbb{R})$ and $f(s)=\sum_{n=1}^\infty b_n n^{-s}$ is in $\Hp$ for some $1\le p \le \infty$, then
\begin{equation}\label{saksmidenti} \Vert P_{\psi}f \Vert_{\Hp}\leq \Vert f\Vert_{\Hp}\,\Vert \psi\Vert_1.\end{equation}
\end{lemma}
\begin{proof} By \eqref{saksident} and the vertical translation invariance of the norm in $\Hp$, we get 
\[ \Vert P_{\psi} f\Vert_{\Hp}\leq\int_{-\infty}^\infty \Vert T_{t} f\Vert_{\Hp}\,|\psi(t)| dt=\Vert f\Vert_{\Hp}\,\Vert \psi\Vert_1.\]

\end{proof}
We now turn to Saksman's alternate proof of Theorem~\ref{parsum}. The idea is to choose $\psi$ so that $\widehat{\psi}$ is smooth and $P_\psi f$ is a good approximation to $S_N f$. A suitable trade-off between these requirements can be made as follows.
Let $\Delta_h$ be the indicator function of the interval $[-h, h]$ for $h>0$, and let $\Lambda=(N/2)(\Delta_1\ast \Delta_{1/N})$ be the even trapezoidal function with nodes at $1-1/N$ and $1+1/N$. Then $\Lambda(x)=1$  for $x$ in $[-(1-1/N),\ 1-1/N]$. We see that 
$$\widehat {\Lambda}(\xi)=(N/2)\, \widehat {\Delta_1}(\xi) \widehat {\Delta_{1/N}}(\xi)=(N/2)\, \frac{\sin \xi\sin \xi/N}{\xi^2},$$
which implies that $\Vert \widehat{\Lambda}\Vert_1\ll \log N$.  We choose $\psi$ by requiring that $\widehat{\psi}(x)=\Lambda(x/\log N)$. Then, by the Fourier inversion formula, we have $\| \psi \|_1\ll \log N$. 
In addition, we observe that $P_\psi f$ differs from $S_{N}f=\sum_{n=1}^N b_n n^{-s}$ by at most $$O\big[N(e^{\log N/N} -e^{-\log N/N})\big]=O(\log N)$$ terms that are all of size $\ll \Vert f\Vert_{\Ho}$ since $|b_n|\leq \Vert f\Vert_{\Ho}$. In view of \eqref{saksmidenti}, this ends the second proof of Theorem~\ref{parsum}.

\subsection{The Bohr lift.} Let $\varphi$ be a function in $\GH$ with $c_0=0$, and assume that $\overline{\varphi(\C_0)}$ is a bounded subset of $\C_{1/2}$. We define $\Delta:\C_{1/2}\to \D^{\infty}\cap \ell^2$ by requiring 
$\Delta(s):=(p_{j}^{-s}),$ where $(p_j)$ denotes the sequence of primes. Then there exists a unique analytic map $\Phi:\D^{\infty}\cap \ell^2\to \C_{1/2}$ such that 
\begin{equation}\label{lifting} \Phi\circ \Delta=\varphi.\end{equation}
We call $\Phi$ the Bohr lift of $\varphi$; we  set $\Phi^{\ast}(z):=\lim_{r\to 1}\Phi(rz)$, where $z$ is a point in $\T^\infty$. We denote by $m_\infty$ the Haar measure of $\T^\infty$ and define the pullback measure $\mu_{\varphi}$  by 
\begin{equation}\label{ollssak}\mu_{\varphi}(E):=m_{\infty}(\{z\in \T^\infty : \Phi^{\ast}(z)\in E\  \})=m_{\infty}\left((\Phi^{\ast})^{-1}(E\ )\right).\end{equation}
We will need the following basic result about the Bohr lift  \cite{QS1}.
\begin{theorem}\label{blift} Suppose that $\varphi$ is in $ \GH$, $c_0=0$, and that $\overline{\varphi(\C_0)}$ is a bounded subset of $\C_{1/2}$. Then, for every $f$ in $\Hp$, $1\le p <\infty$, we have
 $$\Vert C_{\varphi}(f)\Vert_{{\Hp}}^{p}=\int_{\T^\infty} \vert f(\Phi^{\ast}(z))\vert^{p}dm_{\infty}(z)=\int_{\overline{\varphi(\C_0)}} \vert f(s)\vert^{p}d\mu_{\varphi}(s).$$
\end{theorem}

\subsection{Type and cotype of $(\Hp)^*$}
Let $(\varepsilon_j)$ denote a Rademacher sequence and $\mathbb{E}$ the expectation. We recall that a Banach space $X$ is of type $p^{\ast}$ with $1\leq p^{\ast}\leq 2$ if, for some constant $C\geq 1$,
$$\mathbb{E}\left\Vert \sum \varepsilon_j x_j\right\Vert\leq C \left(\sum\Vert x_j\Vert^{p^{\ast}}\right)^{1/p^{\ast}}$$ for every finite sequence $(x_j)$ of vectors from $X$. The smallest constant $C$, denoted by $T_{p^\ast}(X)$, is called the type $p^\ast$-constant of $X$.
Similarly,  a Banach space $Y$ is said to be of cotype $q^{\ast}$ with $2\leq q^{\ast}\leq \infty$  if 
$$\mathbb{E}\left\Vert \sum \varepsilon_j y_j\right\Vert\geq C^{-1} \left(\sum\Vert y_j\Vert^{q^{\ast}}\right)^{1/q^{\ast}}$$ for every finite sequence $(y_j)$ of vectors of $Y$. The smallest constant $C$, denoted by $C_{q\ast}(Y)$, is called the cotype $q^\ast$-constant of $Y$.

As already mentioned, the space $X=\Hp$ is well understood as the subspace $H^{p}(\T^\infty)$ of $L^{p}(\T^\infty)$, but it is not complemented in $L^{p}(\T^\infty)$. This means that its dual $Y$ is something rather mysterious. But the following fact will be sufficient for our purposes. 
\begin{lemma}\label{type} The Banach space $Y=({\Hp})^{\ast}$ is of cotype $\  \max (q,2)$ where $q$ is the conjugate exponent of $p$. 
\end{lemma}
\begin{proof} $\Hp$ is isometric to the subspace $H^{p}(\T^\infty)$ of $L^{p}(\T^\infty)$. The latter space is of type $p^{\ast}=\min (p,2)$ by Fubini's theorem and Khintchin's inequality. According to a result of Pisier (from \cite{PIS}, see also  \cite{MAUPIS}, \cite[p. 220]{DIJATO}, or \cite[p. 165]{LIQU}), its dual $Y$ is of cotype $q^{\ast}=\max(q,2)$\quad  ($1/p+1/q=1$)  and, moreover,  
\begin{equation}\label{gipi}C_{q\ast}(Y)\leq 2 T_{p^\ast}(\Hp).\end{equation} Hence the conclusion of the lemma follows.\end{proof}

\subsection{Description of  the spectrum of compact composition operators}\label{Bayarts}
 A complete description of the eigenvalues of $C_\varphi$ on $\Ht$ was given in  \cite{BAYA}. We will  show that this result is valid for  $\Hp$ as well:

\begin{theorem}\label{fred} Let $\varphi(s)=c_{0}s+\sum_{n=1}^\infty c_n\, n^{-s}$ induce a compact operator on ${\Hp},\ 1\leq p<\infty$. Then, the eigenvalues of $\ C_\varphi$ have multiplicity one and  are 
\begin{itemize}
\item[(a)] $\lambda_n=[\varphi'(\alpha)]^{n-1}, \quad n=1,2,\ldots,$ when $c_0=0$, where $\alpha$ is the fixed point of $\varphi$ in $\C_{1/2}$. 
\item[(b)] $\lambda_n=n^{-c_1},\quad n=1,2,\ldots, $ when $c_0=1$.

\end{itemize}
\end{theorem}
 \begin{proof}  We adapt the proof of \cite{BAYA} to $\Hp$. To begin with, we set $E:=\big\{[\varphi'(\alpha)]^{n-1}, \ n=1,2,\ldots\big\}$ and let $\sigma(C_\varphi)$  be the spectrum of $C_\varphi$. For a given integer $m\geq 1$, consider the space $\mathcal{K}_m$ defined by 
 $$\mathcal{K}_m=\hbox{\ span}\ (\delta_\alpha, \delta'_{\alpha}, \cdots,\delta^{(m)}_{\alpha}) \subset (\Hp)^{\ast},$$
 where by definition $\delta^{(k)}_{\alpha}(f)=f^{(k)}(\alpha)$. This space is invariant under $C_{\varphi}^{\ast}$ since $\varphi(\alpha)=\alpha$. It is also finite-dimensional, and the matrix of $(C_{\varphi}^{\ast})_{\mid \mathcal{K}_m}$ on the natural basis of $\mathcal{K}_m$ is upper triangular with diagonal elements $[\varphi'(\alpha)]^{n-1},\ 1\leq n\leq m+1$. This shows that $E\subset \sigma(C_\varphi^{\ast})=\sigma(C_\varphi)$.  
  For the reverse inclusion, we use the compactness of $C_\varphi$ and  
K\"onigs's theorem \cite[pp. 90--91]{Shap-livre} which in particular claims that $f\circ\varphi=\lambda\,f$ with $f$ a non-zero analytic  function in $\C_{1/2}$ implies that $\lambda=[\varphi'(\alpha)]^k$ for some non-negative integer $k$ and that the corresponding eigenfunctions  generate a one-dimensional space of $\Hp$.
 
 As for (b), set $E=\{n^{-c_1},\ n\geq 1\}$. The proof of \cite{BAYA} still gives $\ \sigma(C_\varphi)\subset E\cup\{0\}$  for $\Hp$. Moreover, for a fixed integer $m$, the  vector spaces $\mathcal{K}_m$ and $\mathcal{L}_m$ respectively generated by $1,2^{-s},\ldots, m^{-s}$ and $j^{-s},\ j>m$ are complementary, $\mathcal{L}_m$ is stable by $C_\varphi$, $\mathcal{K}_m$  is finite-dimensional, and the matrix of   ${C_\varphi}_{|\mathcal{K}_m}$  is lower triangular with diagonal elements $j^{-c_1}, 1\leq j\leq m$. Therefore, $E\subset \sigma(C_\varphi)$ as in  \cite{BAYA}. 
    \end{proof} 
The result of \cite{BAYA} stating that $\sigma(C_\varphi)=\{0,1\}$ when $c_0\geq 2$ also extends from $\Ht$ to $\Hp$, but this case will not be needed in this work and is omitted here.

\subsection{Vertical translates}

$\mathbb T^\infty$ may be identified with the dual group of $\mathbb Q_+$, where $\mathbb Q_+$ denotes the multiplicative discrete group of strictly positive rational numbers: given a point $z=(z_j)$ on $T^\infty$, we define a character $\chi$ on $\mathbb Q_+$ by its values at the primes by setting
$$\chi(2)=z_1,\ \chi(3)=z_2,\ \dots,\ \chi(p_m)=z_m,\ \dots$$
and by extending the definition multiplicatively. In the sequel, we will associate the character $\chi$ with the point $(z_j)$ and refer to $\chi$ as well as a point on $\T^\infty$. In particular, we will be interested in properties that hold for almost all characters $\chi$ with respect to the Haar measure $m_\infty$ on $\T^\infty$.


In \cite{HLS} and \cite{BAY}, it was explained that it is useful to consider $f_\chi$ to get a more profound understanding of the function theoretic properties of $f$. For example, for almost all characters $\chi$ , the function $f_\chi$ can be extended to $\mathbb C_0$. Moreover, we obtain another way to compute the norm of $f$ (see \cite[theorem 4.1]{HLS} or \cite[lemma 5]{BAY}):
\begin{lemma}\label{LEMCALCULNORME}
Let $\mu$ be a finite Borel measure on $\mathbb R$. Then :
$${\|f\|_{{\Hp}}^{p}\,\mu(\mathbb R)}=\int_{\mathbb T^\infty}\int_{\mathbb R} |f_{\chi}(it)|^pd\mu(t)dm_\infty(\chi).$$
\end{lemma}

We shall need to extend our notation of vertical translates to the class of functions of the form $\varphi(s)=c_0s+\psi(s)\in\mathcal G$. For such functions, $\varphi_\chi$ will mean 
$$\varphi_\chi(s)=c_0s+\psi_\chi(s).$$
The connection between the composition operators $C_{\varphi}$ and $C_{\varphi_\chi}$ is clarified in \cite{GORHED}, where it is proved that  for every $f$ in $\Hp$ and every $\chi$ in $\T^\infty$, 
$$(f\circ\varphi)_\chi(s)=f_{\chi^{c_0}}\circ\varphi_\chi(s),\ s\in\mathbb C_{1/2},$$
where $\chi^{c_0}$ is the character taking the value $[\chi (n)]^{c_0}$ at $n$.
Moreover, for almost all $\chi$ in $\mathbb T^\infty$, this relation remains true in $ \mathbb C_0$.

\section{Carleson measures and sequences, and interpolating sequences}
\subsection{The case $p<\infty$}

 We will denote by $\delta_s$ the functional of point evaluation on ${\Hp}$ at the point $s=\sigma+it$ in $\C_{1/2}$, so that  $\delta_{s}(f)=f(s)$. The norm of $\delta_s$ was computed by Cole and Gamelin \cite{COGA} (see also \cite{BAY}): 
\begin{equation}\label{norm} \Vert \delta_{s}\Vert=[\zeta(2\sigma)]^{1/p}.\end{equation}
It should be noted that $\zeta(2\sigma)\approx (2\sigma-1)^{-1}$ when $s$ is restricted to a set on which  $\sigma$ is uniformly bounded. 

Let now  $\mu$ be a nonnegative Borel measure on the half-plane $\C_{1/2}$. We say that $\mu$ is a Carleson measure for $\Hp$  if 
there exists a positive constant $C$ such that
\[ \int_{\C_{1/2}} |f(s)|^{p} d\mu(s) \le C \| f\|_{\Hp}^p \]
holds for every $f$ in $\Hp$. The smallest possible $C$ in this inequality is called the $\Hp$ Carleson norm of $\mu$. We denote it by 
$\|\mu\|_{\mathcal{C}, \Hp}$  and declare that $\|\mu\|_{\mathcal{C}, \Hp}=\infty$ if $\mu$ fails to be a Carleson measure for $\Hp$. 
  Let next $S=(s_j)$ be a  sequence of distinct points of $\C_{1/2}$. We say that this sequence is a Carleson sequence if
the discrete measure  
$$\mu_S=\sum_{j} \frac{\delta_{s_j}}{\Vert \delta_{s_j}\Vert^{p}}$$ is a Carleson measure for ${\Hp}$ in the above sense. The Carleson $\Hp$ norm $\|\mu_S\|_{\mathcal{C}, \Hp}$ is called the Carleson $\Hp$ constant of $S$. 

The will need the following estimate. 

\begin{lemma}\label{firstlem}
Let $\mu$ be a nonnegative Borel measure in the half-plane $\C_{1/2}$ whose support is contained in $\overline{\C_{\theta}}$ for some $\theta>1/2$.
Then 
\[ \|\mu\|_{\mathcal{C}, \Hp} \le \begin{cases} [\zeta(2\theta)]^{(p-2)/p} \|\mu\|_{\mathcal{C}, \Ht} , & p\ge 2 \\
 [\zeta(2\theta)] \|\mu\|, & 1\leq p<2. \end{cases} \]
\end{lemma}

\begin{proof}
For an arbitrary point $s=\sigma+it$ in the support of $\mu$, we have
\[ |f(s)|\le [\zeta(2\sigma)]^{1/p} \| f\|_{\Hp}\le [\zeta(2\theta)]^{1/p} \| f\|_{\Hp}. \]
We infer from this that
\[ \int_{\C_{1/2}} |f(s)|^{p} d\mu(s)=\int_{\C_{1/2}} |f(s)|^{p-2}|f(s)|^2 d\mu(s) \le [\zeta(2\theta)]^{(p-2)/p} \| f\|_{\Hp}^{p-2} \int_{\C_{1/2}} |f(s)|^{2} d\mu(s). \]
Now the result follows since $\| f\|_{\Ht} \le \| f\|_{{\Hp}}$ when $p\geq 2$. The (poor) estimate in the case $p<2$ is obvious. \end{proof}

The $\Hp$ constant of interpolation $M_{{\Hp}}(S)$ is defined as  the infimum of the constants $K$ with the following property: for every  sequence $(a_j)$ of complex numbers such that 
\[ \sum_{j}\vert a_j\vert^p \Vert \delta_{s_j}\Vert^{-p}<\infty, \]  there exists a function $f$ in $\Hp$ such that
$$f(s_j)=a_j \ \ \text{for all $ j$ and}\  \ \Vert f\Vert_p\leq K\left(\sum_{j}\vert a_j\vert^p \Vert \delta_{s_j}\Vert^{-p}\right)^{1/p}.$$
For a sequence $S$ of distinct points in $\mathbb C_0$, we will denote by $M_{{\Hi}}(S) $ the best constant $K$ such that, for every bounded sequence $(a_j)$ of complex numbers, there exists a function $f$ in $\Hi$ such that 
$$f(s_j)=a_j \ \ \text{ for all $ j$ and}\ \ \Vert f\Vert_\infty\leq K\,\sup_{j}\vert a_j\vert .$$   

In the next lemma, we will use the following notation. Given a sequence $S=(\sigma_j+i t_j)$ and a real number $\theta$, we write $S+\theta := (\sigma_j+\theta+i t_j)$.


\begin{lemma}\label{secondlem}
Suppose that $\theta>1/2,\ \delta>0$ and that $S=(s_j=\sigma_j+it_j)_{j=1}^n$ is a finite sequence in the half-plane $\C_{1/2+\delta}$. Then  
\[
M_{\Hp} (S) \le  [\zeta(2\theta)]^{1/\min(2,p)}\left(\frac{\zeta(1+2\delta)}{\zeta(1+2(\delta+\theta))}\right)^{1/\min(2,p)}\, n^{1/\min(2,p)-1/p}   \big(M_{\Ht}(S+\theta)\big)^{2/\min(p,2)}
\]
for $1\le p \le \infty$. 
\end{lemma}

\begin{proof}
Given a sequence $(a_j)_{j=1}^n$, we find the minimal norm solution $F$ to the interpolation problem $F(s_j+\theta)=a_j$ in $\Ht$. From the definition of the constant of interpolation and \eqref{norm}, we get the
basic estimate
\begin{equation} \label{base} \| F \|_{\Ht} \le M_{\Ht}(S+\theta) \left(\sum_{j=1}^n |a_j|^2 [\zeta(2(\sigma_j+\theta))]^{-1}\right)^{1/2}. \end{equation}
By our restriction on $S$ and the fact that $\zeta'/\zeta$ increases on $(1,\infty)$, we have 
\[ \frac{\zeta(2\sigma_j)}{\zeta(2(\sigma_j+\theta))}\le \frac{\zeta(1+2\delta)}{\zeta(1+2(\delta+\theta))}, \]
which gives
\[ \| F \|_{\Ht} \le  \left(\frac{\zeta(1+2\delta)}{\zeta(1+2(\delta+\theta))}\right)^{1/2} \, M_{\Ht}(S+\theta) 
\left(\sum_{j=1}^n |a_j|^2 [\zeta(2\sigma_j)]^{-1}\right)^{1/2} \]
when inserted into \eqref{base}. 
We first assume $1\le p<2$. Using \eqref{norm}  which gives $|a_j|\leq \sqrt{\zeta(2\theta)}\Vert F\Vert_{2}$, we then get
\begin{eqnarray*}
 \| F \|_{\Ht} &\le& \left(\frac{\zeta(1+2\delta)}{\zeta(1+2(\delta+\theta))}\right)^{1/2} M_{\Ht}(S+\theta) [\zeta(2\theta)]^{1/2-p/4} \| F\|_{\Ht}^{1-p/2}\\
&&\quad\times  \left(\sum_{j=1}^n |a_j|^p [\zeta(2\sigma_j)]^{-1}\right)^{1/2}.
\end{eqnarray*}
This yields
\begin{eqnarray*}
 \| F \|_{\Ht} &\le& \left(\frac{\zeta(1+2\delta)}{\zeta(1+2(\delta+\theta))}\right)^{1/p} \big(M_{\Ht}(S+\theta) \big)^{2/p} [\zeta(2\theta)]^{1/p-1/2} \\
&&\quad\times  \left(\sum_{j=1}^n |a_j|^p\|\delta_{s_j}\|^{-p}\right)^{1/p}.
\end{eqnarray*}
When $2\le p<+\infty$, we simply use H\"{o}lder's inequality to obtain
\begin{eqnarray*}
 \| F \|_{\Ht} &\le& \left(\frac{\zeta(1+2\delta)}{\zeta(1+2(\delta+\theta))}\right)^{1/2} \, M_{\Ht}(S+\theta) n^{1/2-1/p}\\ &&\quad\times\left(\sum_{j=1}^n |a_j|^p\|\delta_{s_j}\|^{-p}\right)^{1/p}.
\end{eqnarray*}
For $p=\infty$, we get 
\[ \| F \|_{\Ht}\le M_{\Ht}(S+\theta) n^{1/2} \sup_j |a_j|.\]
We now observe that the shifted function $G(s)=F(s+\theta)$ is in $\Hi$ since $\theta>1/2$, and that $G(s_j)=a_j$ as well.
Hence
$$\Vert G\Vert_{{\Hp}}\leq \Vert G\Vert_{{\Hi}}\leq\sqrt{\zeta(2\theta)} \Vert F\Vert_{\Ht}.$$

 \end{proof}

The preceding lemma is useful because we have good estimates for constants of interpolation in the $\Ht$ setting, thanks to the following key result from \cite{QS1}. Here we use the notation $S_R$ for the subsequence of points $s_j$ from $S$ that satisfy $|\Imag s_j|\leq R$.

\begin{lemma}\label{secondlem2}
Suppose $S=(s_j=\sigma_j+it_j)$ is an interpolating sequence for $H^2(\C_{1/2})$ and that there exists a number $\theta>1/2$ such that $1/2<\sigma_j \le \theta$ for every $j$. Then there exists a constant $C$,  depending on $\theta$, 
such that
\begin{equation}\label{above} M_{\Ht} (S_R) \le  C [M_{H^2(\C_{1/2})}(S)]^{2\theta+6} R^{2\theta+7/2} \end{equation}
whenever $R\ge \theta+1$. 
\end{lemma}

This result is based on \cite{S} and relies on quite involved estimates for solutions of the $\overline{\partial}$ equation. 

We mention finally the remarkable fact that, in general 
\begin{equation}\label{remarkable}
M_{\Ho}(S)\le [M_{\Ht}(S)]^2 
\end{equation}
whenever $S$ is an interpolating sequence for $\Ht$. This bound follows from the observation (see \cite{OLSA}) that we may solve $f(s_j)=a_j$ in $\Ho$ by first solving $g(s_j)=\sqrt{a_j}$ in $\Ht$ and then setting $f=g^2$. Inequality \eqref{remarkable} is the reason $\Ho$ stands out as a distinguished case in our context; for other values of $p$, we do not know how to obtain a nontrivial bound for $M_{\Hp}(S)$ when $\Real s_j \to 1/2$, and essentially nothing is known about the $\Hp$ interpolating sequences.


\subsection{ The case $p=\infty$}
If $H^{\infty}$ denotes the  Banach space of bounded analytic functions on  $\C_0$ and if  $S=(s_j)$ is a sequence of points in $\C_0$, we define the interpolation constant $M_{H^\infty}(S)$ as the infimum of constants $C$ such that for any bounded sequence $(a_j)$, the interpolation problem   $a_j=f(s_j),\ j=1,2,\ldots,$   has a solution $f\in H^\infty$  such that $\Vert f\Vert_\infty\leq C \sup_{j}|a_j|$.
We will make use of  a result of the third-named author \cite{SE},  which can be rephrased as follows: 
\begin{theorem}\label{krisei} Let $S$ be a  subset of $\ \C_{0}$, bounded by $K$. Then, there exists a  positive constant  $\gamma$, depending only on $K$, such that 
\begin{equation}\label{tsf} M_{{\Hi}}(S)\ll \big[M_{H^\infty}(S)\big]^{\gamma}.\end{equation}
\end{theorem}

This result is implicit in \cite{SE}; it is obtained by combining the interpolation theorem of Berndtsson and al. \cite{BECHLI} (see \cite[Lemma~3]{SE}) with \cite[Lemma~4]{SE}.  


\section{A Littlewood--Paley formula for $\Hp$ and proof of Theorem~\ref{sufficientcompact}}

Our new Littlewood--Paley formula reads as follows (abbreviating $\Vert\quad\Vert_{\Hp}$ to  $\Vert\quad\Vert_{p}$.) 
\begin{theorem}\label{PROPLITTLEWOODPALEY}
Let $\mu$ be a probability measure on $\mathbb R$ and $p\ge 1$. Then
\begin{equation}\label{lpf}\|f\|_p^p\asymp |b_1|^p+\int_{\mathbb T^\infty}\int_{0}^{+\infty}\int_{\mathbb R}\sigma |f_\chi(s)|^{p-2}|f_\chi'(s)|^2 d\mu(t)d\sigma dm_{\infty}(\chi)\end{equation}
holds for every Dirichlet series $f(s)=\sum_{n\geq 1}b_nn^{-s}$ in  $\Hp$ .
\end{theorem}
The notation $u(f)\asymp v(f)$ means as usual that there exists a constant $C\geq 1$ such that for every $f$ in question, $C^{-1}u(f)\leq v(f)\leq Cu(f)$.
\begin{proof}
We start from the Littlewood--Paley formula for $H^p(\mathbb D)$, which appears for instance in \cite{Ya79}: We have   
$$\|g\|_{H^p(\mathbb D)}^p\asymp |g(0)|^p+\int\!\int_{\mathbb D}{(1-|z|^2)} |g(z)|^{p-2} |g'(z)|^2 d\lambda(z)$$
when $g$ is in $H^p(\D)$, where now $d\lambda$ denotes Lebesgue area measure on $\D$.
Next we let $f$ be a Dirichlet polynomial, $\xi>0$, and consider the Cayley transform
$$\omega_{\xi}(z)=\xi\frac{1+z}{1-z},\quad \omega_\xi^{-1}(s)=\frac{s-\xi}{s+\xi}.$$
By Lemma \ref{LEMCALCULNORME},
\begin{equation}\label{basicf} \|f\|_p^p=\int_{\mathbb T^\infty}{\left(\int_{\mathbb R}|f_\chi(it)|^p\frac{\xi}{\pi(\xi^2+t^2)}dt\right)dm(\chi)}.\end{equation}
For fixed $\chi$ on $\mathbb T^\infty$, we find that
\begin{eqnarray*}
\int_{\mathbb R}|f_\chi(it)|^p\frac{\xi}{\pi(\xi^2+t^2)}dt&=&{\Vert f_{\chi}\circ\omega_\xi\Vert_{H^p(\mathbb D)}^p} \\
&\asymp&|f_\chi(\xi)|^p+\int\!\int_{\mathbb D}(1-|z|^2)|f_\chi\circ\omega_\xi(z)|^{p-2} |f_\chi'\circ\omega_\xi(z)|^2 |\omega_\xi'(z)|^2 d\lambda(z).
\end{eqnarray*}
By using the change of variables $s=\sigma+i t=\omega_\xi(z)$, we get
\begin{eqnarray*}
\int_{\mathbb R}|f_\chi(it)|^p\frac{\xi}{\pi(\xi^2+t^2)}dt&\asymp&|f_\chi(\xi)|^p+\int_{0}^{+\infty}\int_{\mathbb R}
\left(1-\frac{|s-\xi|^2}{|s+\xi|^2}\right)|f_\chi(s)|^{p-2}|f_\chi'(s)|^2dtd\sigma\\
&\asymp&|f_\chi(\xi)|^p+\int_{0}^{+\infty}\int_{\mathbb R}\frac{\sigma \xi}{(\sigma+\xi)^2+t^2}|f_\chi(s)|^{p-2}|f_\chi'(s)|^2dt d\sigma.
\end{eqnarray*}
We integrate this over $\mathbb T^\infty$. In view of \eqref{basicf}, this gives 
\begin{eqnarray*}
\|f\|_p^p&\asymp&\int_{\mathbb T^\infty}|f_\chi(\xi)|^pdm_\infty(\chi)+\\
&&\quad\quad \int_{0}^{+\infty}\frac{\sigma \xi}{\sigma+\xi}
\int_{\mathbb T^\infty}\int_{\mathbb R} |f_\chi(s)|^{p-2}|f_\chi'(s)|^2  \frac{\sigma+\xi}{\pi\big((\sigma+\xi)^2+t^2\big)}dtdm_\infty(\chi) d\sigma.
\end{eqnarray*}
{Using again} Lemma \ref{LEMCALCULNORME}, {which} is valid for any finite Borel measure, we obtain that
$$\int_{\mathbb T^\infty}\int_{\mathbb R} |f_\chi(s)|^{p-2}|f_\chi'(s)|^2  \frac{\sigma+\xi}{\pi\big((\sigma+\xi)^2+t^2\big)}dtdm_\infty(\chi)=\int_{\mathbb T^\infty}\int_{\mathbb R}|f_\chi(s)|^{p-2}|f_\chi'(s)|^2  d\mu(t)dm_\infty(\chi).$$
Finally, we conclude  by letting $\xi$ tend to infinity.
\end{proof}

In the sequel, it will be convenient to set $f(+\infty):=b_1$ in \eqref{lpf}. 

\begin{proof}[Proof of Theorem~\ref{sufficientcompact}]
Let $(f_n)$ be a sequence in $\Hp$ that converges weakly to zero. We will let $f_{n,\chi}$ denote the vertical limit function of $f_n$ with respect to the character $\chi$. By assumption (a) of Theorem~\ref{sufficientcompact}, there exists a positive number  $A$ such that $|\Imag \psi|\leq A$. This implies  that $|\Imag \psi_\chi|\leq A$ for any $\chi$ on $\T^\infty$. By the Littlewood--Paley formula, setting $w=u+iv$, we get
\begin{eqnarray*}\|C_\varphi(f_n)\|^p_p & \asymp &  |C_\varphi f_n(+\infty)|^p  \\  & + & \int_{\mathbb T^\infty}\int_0^{+\infty}\int_0^1 u |f_{n,\chi^{c_0}}(\varphi_\chi(w))|^{p-2} |f_{n,\chi^{c_0}}'(\varphi_\chi(w))|^2 |\varphi_\chi'(w)|^2 dv\, du\, dm_{\infty}(\chi).\end{eqnarray*}
Our assumption on $f_n$ implies that  $|f_n(+\infty)|$ and hence $C_\varphi f_n(+\infty)$ tend to zero. In the innermost integral, we use the non-univalent change of variables $s=\sigma +it=\varphi_{\chi}(u+iv)$. Observe that, for every $t$ in $[0,1]$, $-A\leq \Imag s\leq A+c_0$, whence
$$\|C_\varphi(f_n)\|^p_p\ll o(1)+\int_{\mathbb T^{\infty}}\int_0^{+\infty}\int_{-A}^{A+c_0}|f_{n,\chi^{c_0}}(s)|^{p-2}|f_{n,\chi^{c_0}}'(s)|^2\mathcal N_{\varphi_{\chi}}(s)dt\, d\sigma\, dm_{\infty}(\chi).$$
We now use assumption (b) of Theorem~\ref{sufficientcompact} in the following way. For any given $\varepsilon>0$, we let  $\theta>0$ be such that  $\mathcal N_{\varphi}(s)\leq\varepsilon \Real s$ whenever $\Real s<\theta$. We split
the integral over $\mathbb R_+$ into $\int_0^{\theta}+\int_{\theta}^{+\infty}$. For the first integral, say $I_0:=\int_0^\theta$, we use that $\mathcal N_{\varphi_\chi}(s)\leq\varepsilon \Real s$ for any $\chi$ on $\mathbb T^\infty$ and any $s$ with $\Real s <\theta$ (see \cite[Proposition 4]{BAYA}). Using again the Littlewood--Paley formula, we get that there exists some constant $C>0$ such that
$$I_0=\int_{\mathbb T^{\infty}}\int_0^{\theta}\int_{-A}^{A+c_0}|f_{n,\chi^{c_0}}(s)|^{p-2}|f_{n,\chi^{c_0}}'(s)|^2\mathcal N_{\varphi_{\chi}}(s)dtd\sigma dm_{\infty}(\chi)\leq C\varepsilon \|f_n\|_p^p.$$
For the second integral, say $I_\infty:=\int_\theta^\infty$, we observe that $\mathcal N_{\varphi_\chi}(s)\leq\frac 1{c_0}\Real s$ (see \cite[Proposition 3]{BAY}), so that
\begin{eqnarray*}
I_{\infty} & = & \int_{\mathbb T^{\infty}}\int_{\theta}^{+\infty}\int_{-A}^{A+c_0}|f_{n,\chi^{c_0}}(s)|^{p-2}|f_{n,\chi^{c_0}}'(s)|^2\mathcal N_{\varphi_{\chi}}(s)dtd\sigma dm_{\infty}(\chi)\\
& \le &  \frac{1}{c_0}\int_{\mathbb T^{\infty}}\int_{\theta}^{+\infty}\int_{-A}^{A+c_0}\sigma |f_{n,\chi^{c_0}}(s)|^{p-2}|f_{n,\chi^{c_0}}'(s)|^2 dtd\sigma dm_{\infty}(\chi)\\
&= &  \frac{1}{c_0}\int_{\mathbb T^{\infty}}\int_{\theta}^{+\infty}\int_{-A}^{A+c_0}\sigma |f_{n,\chi}(s)|^{p-2}|f_{n,\chi}'(s)|^2 dtd\sigma dm_{\infty}(\chi)\\
& \le & \frac{1}{c_0}\int_{\mathbb T^{\infty}}\int_{\theta/2}^{+\infty}\int_{-A}^{A+c_0}(\sigma+\theta/2) |f_{n,\chi}(s+\theta/2)|^{p-2}|f_{n,\chi}'(s+\theta/2)|^2 dtd\sigma dm_{\infty}(\chi)\\
& \le & \frac{2}{c_0}\int_{\mathbb T^{\infty}}\int_{\theta/2}^{+\infty}\int_{-A}^{A+c_0}\sigma |f_{n,\chi}(s+\theta/2)|^{p-2}|f_{n,\chi}'(s+\theta/2)|^2 dtd\sigma dm_{\infty}(\chi)\\
& \ll & \|f_n(\cdot+\theta/2)\|_p^p,
\end{eqnarray*}
and this last quantity goes to zero since the horizontal translation operator  $f(s)\mapsto f(s+\theta/2)$ acts compactly on $\Hp$.
\end{proof}

\section{Two general lower bounds}\label{general}

 We will let $q$ denote the conjugate exponent of $p$. The evaluation $\delta_s$ at $s$ is in $({\Hp})^{\ast}$ and, by \eqref{norm},
 $\Vert \delta_s\Vert=[\zeta(2\Real s)]^{1/p}$
when $p$ is any real number $\geq 1$. Observe that $\delta_s/\Vert \delta_s\Vert$ converges weakly to $0$ as $\Real s\to 1/2$ and that $C_{\varphi}^{\ast}(\delta_s)=\delta_{\varphi(s)}$, so that a necessary condition for compactness of $C_\varphi:{\Hp}\to {\Hp}$ is that 
$$\lim_{\Real s\to 1/2}\frac{\Vert \delta_{\varphi(s)}\Vert}{\Vert \delta_s\Vert}=\lim_{\Real s\to 1/2}\left(\frac{ \zeta(2\Real\varphi(s))}{ \zeta (2\Real s)}\right)^{1/p}=0.$$
It is therefore not surprising to see the latter quotient appearing in our general estimates for $a_n(C_\varphi)$  in the two theorems given below. These results represent two different ways of obtaining lower bounds for the quantities $a_{n}(C_\varphi)$  via respectively $\Hi$ interpolation and $\Hp$ interpolation.

\begin{theorem}\label{belowc2}
Suppose that $\varphi(s)=c_0s+\sum_{n=1}^{\infty} c_n n^{-s}$ determines a compact composition operator $\Cp$ on  ${\Hp},\ 1\leq p<\infty$. 
Let $S=(s_j)$ and $S'=(s_j')$ be finite sets in $\C_{1/2}$, both of  of cardinality $n$, such that $\varphi(s_j')=s_j$ for every $j$.  Then we have
\begin{equation}\label{pepe} a_n(\Cp)\ge \rho_p n^{-(1/\min(2,p)-1/p)} [M_{{\Hi}}(S)]^{-1}\, \|\mu_{S'}\|_{{\mathcal C},{\Hp}}^{-1/p}\, \inf_{1\leq j\leq n} \left(\frac{\zeta(2 \Real s_j)}{\zeta(2 \Real s'_j)}\right)^{1/p}, \end{equation}
where $\rho_p$ is a  constant depending only on $p$.
\end{theorem}

\begin{proof} As already noted, the transpose $C_{\varphi}^{\ast}:({\Hp})^{\ast}\to ({\Hp})^{\ast}$ still verifies in an obvious way  the mapping equation 
$$C_{\varphi}^{\ast}(\delta_a)=\delta_{\varphi(a)},$$  
which was used extensively in our previous work \cite{QS1}.
We will use the  inequality \eqref{stefan} in the form 
  $$ a_{n}(C_\varphi)\geq a_{n}(C_\varphi^{\ast})\geq b_{n}(C_\varphi^{\ast})$$ and minorate the latter quantity. For clarity, we separate the proof into two parts.

\subsection*{Case 1: $p\geq 2$}

  Let $E$ be the space generated by $\delta_{s'_1}, \ldots, \delta_{s'_n}$. This is an $n$-dimensional space.  Let  $L=\sum_{j=1}^n \lambda_j \delta_{s'_j}$ be an element in  the unit sphere $S_E$ of $E$. If $f$ is in $ {\Hp}$, then the Cauchy-Schwarz inequality  in $\C^n$ implies that
$$|L(f)|=\Big|\sum_{j=1}^n \lambda_j f(s'_j)\Big|=\Big|\sum_{j=1}^n \Lambda_j F_j\Big|\leq \Vert \Lambda\Vert_2\, \Vert F\Vert_2,$$ where we have set 
$$\Lambda_j:=\lambda_j \Vert \delta_{s'_j}\Vert \quad \text{and}\quad  F_j:=f(s'_j) \Vert \delta_{s'_j}\Vert^{-1}$$ as well as $$\Lambda:=(\Lambda_1,\ldots,\Lambda_n)  \quad \text{and}\quad  F:=(F_1,\ldots, F_n).$$  Since $L$ is assumed to be in the unit sphere of $E$,
H\"older's inequality now implies that 
\begin{equation}\label{andone}1\leq n^{1/2-1/p}\,\Vert \Lambda\Vert_2 (\Vert\mu_{S'} \Vert_{\mathcal {C},{\Hp}})^{1/p}.\end{equation}

Next we observe that the sequence $\delta_{s_j}$ is unconditional with constant $\leq M_{{\Hi}}(S)=:M_S$, i.e.,
\begin{equation}\label{andnext}M_{S}^{-1}\left\Vert \sum \omega_j \lambda_j \delta_{s_j}\right\Vert \leq \left\Vert \sum  \lambda_j \delta_{s_j}\right\Vert\leq M_S\,\left\Vert \sum \omega_j \lambda_j \delta_{s_j}\right\Vert \end{equation}
for any choice of scalars $\lambda_j$ and unimodular scalars $\omega_j$. To see this, we first set
$$\Phi=\sum_{j=1}^n \lambda_j \delta_{s_j},\  \Phi_\omega=\sum_{j=1}^n \omega_j \lambda_j \delta_{s_j}.$$
If $h$ in $ \Hi$ verifies $h(s_j)=\omega_j,\ 1\leq j\leq n$, and $\Vert h\Vert_{\infty}\leq M_S$, then we see that, for every $f\in {\Hp}$, 
$\Phi_{\omega}(f)=\Phi(hf)$. Since ${\Hi}$ is isometrically  equal to the multiplier algebra of ${\Hp}$, we therefore get that
$$|\Phi_{\omega}(f)|\leq \Vert \Phi\Vert \Vert hf\Vert_{{\Hp}}\leq \Vert \Phi\Vert \Vert h\Vert _\infty \Vert f\Vert_{{\Hp}}\leq M_S\,\Vert \Phi\Vert  \Vert f\Vert_{{\Hp}}.$$
 This gives the left-hand inequality of \eqref{andnext}. The right-hand inequality readily follows, replacing $\lambda_j$ by $\lambda_j \omega_j$ and $\omega_j$ by $\overline{\omega_j}$.  Averaging with respect to independent choices of $\omega_j$ (Rademacher variables) and using Lemma~\ref{type} and \eqref{gipi}, we get from the left-hand side of \eqref{andnext}, setting $\rho_p=[2T_{2}({\Hp})]^{-1}$, that
\begin{eqnarray*} \Vert \Phi\Vert & \geq & M_{S}^{-1}\mathbb{E}\left\Vert \sum \omega_j\,\lambda_j\,\delta_{s_j}\right\Vert\geq M_{S}^{-1} \rho_p\left(\sum|\lambda_j|^{2}|\Vert \delta_{s_j}\Vert^2\right)^{1/2} \\ &\geq & M_{S}^{-1}\rho_p\, \inf_{1\leq j\leq n}\frac{\Vert \delta_{s_j}\Vert}{\Vert \delta_{s'_j}\Vert} \left(\sum|\lambda_j|^{2}|\Vert \delta_{s'_j}\Vert^2\right)^{1/2}. \end{eqnarray*}
By the mapping equation, $C_{\varphi}^{\ast}(L)=\Phi$, and hence we get 
 $$\Vert C_{\varphi}^{\ast}(L) \Vert\geq \rho_p\, M_{S}^{-1} \inf_{1\leq j\leq n} \left(\frac{\zeta(2 \Real s_j)}{\zeta(2 \Real s'_j)}\right)^{1/p} \,\Vert \Lambda\Vert_2.$$
Using \eqref{andone}, we finally obtain 
 \[\Vert C_{\varphi}^{\ast}(L)\Vert\geq \rho_p n^{-(1/2-1/p)}\, M_{S}^{-1}  \Vert \mu_{S'}\Vert_{\mathcal {C},{\Hp}}^{-1/p} \inf_{1\leq j\leq n} \left(\frac{\zeta(2 \Real s_j)}{\zeta(2 \Real s'_j)}\right)^{1/p}. \] 
This implies the desired result since 
$b_{n}(C_{\varphi}^{\ast})\ge \inf_{L\in S_E} \Vert C_{\varphi}^{\ast}(L)\Vert$.   
\subsection*{Case 2: $1\leq p< 2$}

We follow word for word the same route, with H\"older instead of Cauchy--Schwarz and $({\Hp})^{\ast}$ of cotype $q$ (see Lemma \ref{type}). In the special case $p=1$, we have $q=\infty$, but then (\ref{andnext}) implies $$\sup_{1\leq j\leq n}|\lambda_j| \Vert \delta_{s_j}\Vert\leq M_S \Vert \sum_{j} \lambda_j \delta_{s_j}\Vert$$
so that 
$$\inf_{1\leq j\leq n}\frac{\Vert \delta_{s_j}\Vert}{\Vert \delta_{s'_j}\Vert}\Vert\, \Lambda\Vert_\infty \leq\sup_{1\leq j\leq n}|\lambda_j| \Vert \delta_{s_j}\Vert\leq M_S  \Vert C_{\varphi}^{\ast}(L)\Vert.$$

We thus obtain for all $1\leq p<2$ the two inequalities
\begin{eqnarray*} 1 & \leq & \Vert \Lambda\Vert_q (\Vert\mu_{S'} \Vert_{\mathcal {C},{\Hp}})^{1/p}, \\
 \Vert C_{\varphi}^{\ast}(L)\Vert\geq \rho_p M_{S}^{-1}\Vert \Lambda\Vert_q  \inf_{1\leq j\leq n}\frac{\Vert \delta_{s_j}\Vert}{\Vert \delta_{s'_j}\Vert} &\geq  &\rho_p M_{S}^{-1}\, \Vert \mu_{S'}\Vert_{\mathcal {C},{\Hp}}^{-1/p}\, \inf_{1\leq j\leq n} \left(\frac{\zeta(2 \Real s_j)}{\zeta(2 \Real s'_j)}\right)^{1/p},  \end{eqnarray*}
where $\rho_p=(2 T_p({\Hp}))^{-1}$  for $1<p<2$ and $\rho_1=1$.
This takes care of the second case and ends the proof of Theorem \ref{belowc2}. \end{proof}

 We turn to the bound for $a_n(C_{\varphi})$ using $\Hp$ interpolation.  
 
  
\begin{theorem}\label{belowc2bis}
Suppose that $\varphi(s)=c_{0}s+\sum_{n=1}^{\infty} c_n n^{-s}$ determines a compact composition operator $\Cp$ on  ${\Hp}$. Let $S=(s_j)$ and $S'=(s'_j)$ be finite sets in $\C_{1/2}$, both of  of cardinality $n$, such that $\ \varphi(s'_j)=s_j$ for every $j$.  Then 
we have
\begin{equation}\label{pp} a_n(\Cp)\ge n^{-(1/\min(2,p) -1/p)} [M_{{\Hp}}(S)]^{-1}\, \|\mu_{S'}\|_{{\mathcal C},{\Hp}}^{-1/p}\, \inf_{1\leq j\leq n} \Big(\frac{\zeta(2 \Real s_j)}{\zeta(2\Real s'_j)}\Big)^{1/p}. \end{equation}
\end{theorem}
 \begin{proof}
  The proof begins like that of Theorem \ref{belowc2}, using  the Bernstein numbers of the transpose of $C_\varphi:\Hp\to \Hp$. We have once again
  \begin{equation}\label{andonebis}1\leq n^{1/\min(2,p)-1/p}\,\Vert \Lambda\Vert_2 (\Vert\mu_{S'} \Vert_{\mathcal {C},{\Hp}})^{1/p}.\end{equation}

  From now on,  we no longer appeal to cotype and ${\Hi}$ interpolation, but to ${\Hp}$ interpolation and a Boas-type lower bound, namely
\[\Big\Vert \sum_{j} \lambda_j\delta_{s_j}\Big\Vert\geq \big[M_{\mathscr{H}^p}(S)\big]^{-1}\big(\sum_{j}|\lambda_j|^q\,\Vert \delta_{s_j}\Vert^{q} \big)^{1/q}\geq\big[M_{\mathscr{H}^p}(S)\big]^{-1}\Vert\Lambda\Vert_2. \] Here the latter inequality holds since $q\leq 2$ and therefore $\Vert\Lambda \Vert_{\ell^q}\geq \Vert\Lambda \Vert_{\ell^2} $. The first inequality is proved by duality as follows. Write
\[\big(\sum_{j} |\lambda_j|^q\Vert \delta_{s_j}\Vert^{q}\big)^{1/q}=\sum_{j} c_j\lambda_j\Vert \delta_{s_j}\Vert, \quad \text{where}\quad \sum_{j}|c_j|^p=1.\]
Observe that $\sum_{j}\big(|c_j| \Vert \delta_{s_j}\Vert\big)^{p}\Vert \delta_{s_j}\Vert^{-p}=1$ so that $c_j\Vert \delta_{s_j}\Vert=f(s_j)$ for some $f\in {\Hp}$ with norm $\leq M_{\mathscr{H}^p}(S)$. We finally get
\[ \big(\sum_{j} |\lambda_j|^q\Vert \delta_{s_j}\Vert^{q}\big)^{1/q}=\sum_{j} \lambda_j\, f(s_j)=\Phi(f) \leq \Vert f\Vert_{{\Hp}}\|\Phi\| \leq  M_{\mathscr{H}^p}(S)\Big\Vert \sum_{j} \lambda_j\delta_{s_j}\Big\Vert. \] 
Using \eqref{andonebis} and the bound $\|\Phi\|\ge [M_{\mathscr{H}^p}(S)]^{-1}\Vert\Lambda\Vert_2$, we conclude the proof in the same way as we did in the proof of Theorem \ref{belowc2}.
\end{proof}

The difficulty in applying  Theorems \ref{belowc2} or \ref{belowc2bis} is that it is in general difficult to get good estimates for $M_{{\Hi}}(S)$, $M_{{\Hp}}(S)$ or  $\Vert \mu_{S}\Vert_{\mathcal{C}, {\Hp}}$. We will later see some special cases in which this is in fact possible. 

\section{Proof of the bounds in Theorem~\ref{General}}\label{proofbounds}

\subsection*{Part (a)} 

 We present two proofs: a sketchy one, based on the spectral properties of $C_\varphi$, and a detailed one, based on Theorem \ref{belowc2}, illustrating the utility of $\Hi$ interpolation.

  The first approach  uses the spectrum $\sigma(C_\varphi)$ of $C_\varphi$ on $\Hp$ described in Theorem \ref{fred}: 
\begin{equation}\label{spec} \sigma(C_\varphi)=\{0\}\cup\Big\{[\varphi'(\alpha)]^k,\ k=0,1,\ldots\Big\}.\end{equation}
 We can moreover assume that $r_0=|\varphi'(\alpha)|>0$, as in Lemma 6.1 of \cite{QS1}. Finally, \eqref{spec}, Theorem \ref{multi}, and a  tauberian argument show that
$$a_{n}(C_\varphi)\gg r_{0}^{8n}.$$ 

 The second approach goes as follows. Let $\Delta$ be an open disc whose closure is contained in $\C_{1/2}$. Clearly $\varphi(\Delta)$ contains a closed disc $\overline{D}(a,r):=\{s: \ |s-a|\le r\}$ with $0<r<1$. Set 
$$S=\{s_j:=a+r\omega^j : 1\leq j\leq n\},\hbox{\ where}\ \omega=e^{2i\pi/n},$$
and let $S'=\{s'_j\}$  be a set of $n$ distinct points from $\Delta$ such that 
$$\varphi(s'_j)=s_j,\quad 1\leq j\leq n.$$ 
We introduce the associated Blaschke product $$B(s):=\prod_{1\leq j\leq n}\frac{s-s_j}{s+\overline{s_j}-1}=\frac{(s-a)^n -r^n}{(s+\overline{a}-1)^n -r^n} $$ and find, by an elementary computation, that the uniform separation constant $$\delta(S)=\inf_{1\leq j\leq n} (2\sigma_j-1)|B'(s_j)|$$ of $S$ verifies 
\begin{equation}\label{dunil}\delta(S)\gg r^n.\end{equation}
It is known that $M_{H^\infty}(S)\le (2e+4e|\log \delta(S)|)/\delta(S)$ \cite[p. 268]{Ko}, where $M_{H^\infty}(S)$ denotes the constant of interpolation for the space $H^{\infty}(\C_{0})$ of bounded, analytic functions on $\C_0$.


  After having made this choice of $S$ and $S'$, we now estimate each of the three terms appearing on the right-hand side of \eqref{pepe} of Theorem~\ref{belowc2} with help of Theorem \ref{krisei}. We first claim that 
\begin{equation}\label{etun} M_{{\Hi}}(S)\ll r^{-(\gamma+\varepsilon)n}\end{equation}
for every $\varepsilon>0$.
Indeed, it follows from (\ref{dunil}) and the relations between interpolation and uniform separation constants that 
$M_{H^\infty}(S)\ll [1/\delta(S)]^{1+\varepsilon}\ll r^{-(1+\varepsilon)n}$, and then \eqref{tsf} gives the result.
We find next that 
\begin{equation}\label{etdeux} \Vert \mu_{S'}\Vert_{\mathcal{C},{\Hp}}\ll n.\end{equation}
This is a consequence of Lemma \ref{firstlem} since $S'=(s'_j)$ is uniformly bounded  and lies far from the boundary $\Real s=1/2$.  
Finally, we observe that 
\begin{equation}\label{ettrois}\inf_{1\leq j\leq n}\frac{\Vert \delta_{s_j}\Vert}{\Vert \delta_{s'_j}\Vert}\gg 1. \end{equation}
This is immediate since $\Vert \delta_{s_j}\Vert\geq 1$ and $\Vert \delta_{s'_j}\Vert=O(1)$ as $S'$ remains far from the boundary when $n$ increases. 

Now part (a) of Theorem \ref{General} follows from Theorem~\ref{belowc2} if we put together \eqref{etun}, \eqref{etdeux}, and \eqref{ettrois}; taking into account the factors $n$ and $n^{-(1/2-1/p)}$, we observe that we may choose $\delta=r^{\gamma+\varepsilon}$ for an arbitrary $\varepsilon>0$.

\subsection*{Part (b):\ $c_0=1$} 
We first prove  \eqref{bill} by applying Theorem \ref{four} with $r=1/(\Real c_1)$. By  Theorem~\ref{fred}, the left-hand side of \eqref{bernard} is infinite. It follows that the right-hand side is infinite as well, whence the result follows. Our proof of \eqref{bill} given above does not lead to a pointwise estimate of $a_{n}(C_\varphi)=:a_n$. To achieve this,  we use \eqref{albrecht} with $N$ in place of $n$, where $N>2n$ is an integer to be chosen later. We set $\gamma_1=\Real c_1$ and use that $\lambda_{j}(C_\varphi)=j^{-c_1}$ to obtain 
$$2^{-\gamma_1}\,N^{-\gamma_1}=(2N)^{-\gamma_1}\leq e\,\big(a_{1}\cdots a_{N}\big)^{1/N}\leq e\,\big(a_{1}^n\,a_{n}^{(N-n)}\big)^{1/N}= e\,a_{1}^{n/N}a_{n}^{(N-n)/N}.$$
This implies that 
$$a_n\geq 2^{-\gamma_{1}N/(N-n)}\, e^{-N/(N-n)}a_{1}^{-n/(N-n)}N^{-\gamma_{1}N/(N-n)}\gg  N^{-\gamma_{1}N/(N-n)},$$
which we now write as 
$$a_n \gg N^{-\gamma_{1}}\,N^{-\gamma_{1}n/(N-n)}=N^{-\gamma_{1}}\,e^{-\gamma_{1}n\log N/(N-n)}.$$ Choosing $N$ as the integer part of $n\log n+2$ and noting that 
$\frac{n\log N}{N-n}\to 1$, we finally get
$$a_n\gg (n\log n)^{-\gamma_1}$$ as claimed. 

It may be observed that the latter argument gives an alternate proof of  \eqref{bill}.


\subsection*{Part (c):\ $c_0\geq 2$} In this case, the spectrum is reduced to $\{0,1\}$ (see Subsection~\ref{Bayarts}), so that the previous proof does not work. We will proceed differently and use Bernstein numbers and a properly chosen $n$-dimensional space $E$.  This new argument will in fact work also when $c_0=1$ and give an alternate proof for that case.

 Our proof is based on the following lemma which exploits the fact that the collection of linear functions constitute an infinite-dimensional subspace of $H^p(\T^\infty)$. In what follows, we let $\Omega(N)$ denote the number of prime factors in $N$ counted with their multiplicity, and we let $P_{c_0}$ denote orthogonal projection from $\Ht$ onto its subspace generated by the basis vectors $N^{-s}$ with $\Omega(N)= c_0$.
 \bigskip

 \begin{lemma}\label{choice}
Fix an integer $n\geq 1$. Suppose that $\varphi(s)=c_{0}s+\sum_{j=1}^\infty c_j j^{-s}$ is in $\GH$ with $c_0\geq 1$. Let $E$ be the $n$-dimensional subspace of $\Hp$ spanned by the unit vectors $p_{1}^{-s}, p_{2}^{-s},\ldots, p_{n}^{-s}$. Then for every $f(s)=\sum_{k=1}^n b_k p_{k}^{-s}$ in $E$, we have
\begin{equation}\label{projection}P_{c_0}C_{\varphi} f(s)=\sum_{k=1}^n b_k p_{k}^{-c_1} (p_{k}^{c_0})^{-s} \quad \text{and} \quad \Vert P_{c_0} C_{\varphi} f \Vert_{\Hp}\leq \Vert C_{\varphi} f\Vert_{\Hp}.\end{equation} 
  \end{lemma}
\begin{proof} We know from \cite{GORHED} that the following formal computation is allowed to determine the Dirichlet coefficients of $C_{\varphi}(p_{k}^{-s}), 1\leq k\leq n$:
$$C_{\varphi}(p_{k}^{-s})=p_{k}^{-c_0 s} p_{k}^{-c_1}\prod_{j=2}^\infty (1+\sum_{l=1}^\infty \frac{(-c_j \log p_k)^l}{l!}j^{-ls})=:p_{k}^{-c_0 s} p_{k}^{-c_1}(1+\sum_{m\geq 2} \alpha_{k,m} m^{-s}),$$ so that $P_{c_0} C_{\varphi}f$ can be expressed as stated in \eqref{projection}.  For the norm estimate, using the Bohr lift, we note that for $h(s)=\sum_{N=1}^{\infty} \beta_N N^{-s}$, the formula
$$Q(h)(s)=(1/2\pi)\int_{0}^{2\pi} \left(\sum_{N=1}^\infty \beta_N\, e^{i \Omega(N) \theta}N^{-s}\right)  e^{-ic_0 \theta} d\theta$$
defines a norm-one projection from $\Hp$ to its subspace generated by the vectors $N^{-s}$ with $\Omega(N)=c_0$. But this means that
$$\Vert P_{c_0} C_{\varphi} f  \Vert_{\Hp}=\Vert Q C_{\varphi} f  \Vert_{\Hp} \leq \| C_\varphi f\|_{\Hp},$$
which gives the second part of \eqref{projection}. \end{proof}

We are now ready to prove part (c) of Theorem~\ref{General}. Choose $E$ as in Lemma \ref{choice} and let $f$ be a vector in the unit sphere of $E$. By Lemma~\ref{choice} and the Bohr lift , we get
\[ \Vert C_{\varphi}(f)\Vert_{\Hp}  \geq  \big\Vert P_{c_0} C_{\varphi} f \big\Vert_{\Hp}\ge \big\Vert \sum_{k=1}^n c_k p_{k}^{-c_1}z_{k}^{c_0}\big\Vert_{H^{p}(\T^\infty)}=\big\Vert \sum_{k=1}^n c_k p_{k}^{-c_1}z_{k}\big\Vert_{H^{p}(\T^\infty)}, \]
where we for the last relation used the invariance of the Haar measure $m_\infty$ of $\T^\infty$ under the transformation $(z_j)\mapsto (z_{j}^{c_0})$. Applying the Khintchin inequality for the Steinhaus variables $z_j$ twice, we get 
\begin{eqnarray*} \Vert C_{\varphi}(f)\Vert_{\Hp} & \gg & \big\Vert \sum_{k=1}^n c_k p_{k}^{-c_1}z_{k}\big\Vert_{H^{2}(\T^\infty)} 
\ge p_{n}^{-\Real c_1}\big\Vert \sum_{k=1}^n c_k z_{k}\big\Vert_{H^{2}(\T^\infty)} \\ 
& \gg & p_{n}^{-\Real c_1}\big\Vert \sum_{k=1}^n c_k z_{k}\big\Vert_{H^{p}(\T^\infty)}=p_n^{-\Real c_1}\| f\|_{\Hp}. \end{eqnarray*} It follows that 
$$ a_{n}(C_\varphi){\geq} b_{n}(C_\varphi)\gg p_n^{-\Real c_1}.$$
By the Tchebycheff form of the prime number theorem, $p_n\ll n\log n$,  and so the desired estimate follows.


\section{Optimality of the bounds in Theorem~\ref{General}}\label{optimality}
The bounds (a), (b), (c) of Theorem~\ref{General} are optimal in view of the following theorem.
\begin{theorem}\label{optimal} Suppose that $c_0$ is a nonnegative integer and that $\varphi(s)=c_0 s+\sum_{n=1}^\infty c_n n^{-s}$ generates a bounded composition operator $C_{\varphi}$ on ${\Hp}$ for some $1\leq p<\infty$. 
\begin{itemize}
\item [(a)] If $c_0=0$ and $\Omega:=\overline{\varphi(\C_0)}\subset \C_{1/2}$ is compact, then $a_n(C_{\varphi})\ll \delta^n$ for some $0<\delta<1$.
\item [(b)] If $c_0\geq 1$, and if $\varphi(\C_0)\subset \C_A$ for some $A>0$, then $a_n(C_{\varphi})\ll n^{-A}$  if $p>1$ and $a_n(C_{\varphi})\ll (\log n) n^{-A}$ if $p=1$.
 \end{itemize}
 
 \end{theorem}
 \begin{proof} We split the proof into three parts.
 \subsection*{Part (a)} We recall that the $n$th Gelfand number $c_{n}(C_\varphi)$ of  an operator $T$ on $\Hp$ is  
$$c_{n}(T)=\inf_{E} \Vert T|_{E}\Vert,$$ where $E$ runs over all subspaces of $\mathcal{H}^p$ of codimension $<n$.
 Let $E_0$ be the subspace of $\mathcal{H}^p$ defined by 
 $$E_0=\{f\in {\Hp} : f(s_0)=f'(s_0)=\cdots= f^{(n-1)}(s_0)=0\},$$
where $\Real s_0\geq \theta$ and $\theta=\inf_{s\in \Omega} \Real s>1/2$.
This is a subspace of codimension $<n+1$. We will first prove that
 \begin{equation}\label{firstprove}\Vert C_\varphi |_{ E_0}\Vert^{p}\leq \sup_{s\in \Omega}\big\vert B(s)\vert^{p}\zeta(1/2+\theta), \end{equation} where $B$ is the ``adapted" Blaschke product
 $$B(s)=\left(\frac{s-s_0}{s-(1/2+\theta)+\overline{s_0}}\right)^n$$
which is of modulus $1$ on the vertical line $\Real s=1/4+\theta/2$.
We set  $$r:=\sup_{s\in \C_\theta} \Big\vert \frac{s-s_1}{s+\overline{s_1}-1}\Big\vert<1$$  
and $M:=\sup_{s\in \Omega}|B(s)|=r^{n-1}$. 

 We now choose an arbitrary $f$ in $ E_0$. This $f$ can be written $f=Bh$ with $h$ having the same supremum as $f$ on the vertical line $\Real s=1/4+\theta/2$. Using the maximum principle, we observe that 
\begin{eqnarray*} \sup_{s\in \Omega}\vert f(s)\vert^p\leq \sup_{s\in \Omega}\vert B(s)\vert^p\, \sup_{s\in \Omega}\vert h(s)\vert^p & = &M^p\,\sup_{s\in \partial\Omega}\vert h(s)\vert^p=M^p\,\sup_{s\in \partial\Omega}\vert f(s)\vert^p\\ & \leq & M^p\,\zeta(1/2+\theta)\Vert f\Vert^p.\end{eqnarray*}
We use the pullback measure $\mu_{\varphi}$  defined by
\eqref{ollssak}
and the set $\Omega:=\overline{\varphi(\C_0)}$. Using the lifting identity \eqref{lifting} of Theorem \ref{blift}, we then get 
 $$\Vert C_{\varphi}(f)\Vert_{{\Hp}}^{p}=\int_{\T^\infty} \vert f(\Phi^{\ast}(z))\vert^{p}dm_{\infty}(z)=\int_{\Omega} \vert f\vert^{p}d\mu_{\varphi}.$$ 
 We infer from this that
\begin{eqnarray*} \Vert C_{\varphi}(f)\Vert_{\Hp}^p &=& \int_{\T^\infty} \vert f(\Phi^{\ast}(z))\vert^{p}dm_{\infty}(z )= \int_{\Phi^{\ast}(z)\in\Omega}\vert f(\Phi^{\ast}(z))\vert^{p}dm_{\infty}(z) \\ & \leq & \Big[M^p\,\zeta(1/2+\theta)\Big]\Vert f\Vert_{{\Hp}}^{p}, \end{eqnarray*}
which gives \eqref{firstprove}. It follows that 
$$[c_{n}(C_\varphi)]^p\leq \Vert C_\varphi\vert E_0\Vert^{p}\leq \Big[M^p\,\zeta(1/2+\theta)\Big]^{1/p}.$$
Inequality \eqref{stefane}, which states that  $a_{n}(C_\varphi)\leq 2\sqrt n\, c_{n}(C_\varphi)$,  finally gives 
$$a_{n}(C_\varphi)\leq 2\sqrt n\,r^{n-1}[\zeta(1/2+\theta)]^{1/p}.$$

\subsection*{Part (b), p>1}
We consider first the special case in which $\varphi(s)=s+A$ and $A>0$. This function is seen to belong to $ \GH$. For a given integer $n\geq 2$, let $R$ the $(n-1)$-rank operator defined by 
\[ Rf:=\sum_{j=1}^{n-1}j^{-A} x_j e_j,\] 
where $f=\sum_{j=1}^{\infty} x_j e_j$. It follows that $C_{\varphi}f=\sum_{j=1}^{\infty}j^{-A} x_j e_j$ and that 
$$(C_\varphi -R)f=\sum_{j\geq n} j^{-A}x_j\, e_j.$$
Now the contraction principle \eqref{prinzip} with $\lambda_j=j^{-A}$ gives 
$$a_{n}(C_\varphi)\leq\Vert C_\varphi-R\Vert\leq 2Cn^{-A},$$ 
which settles our special case.

 In the general case, we write $\varphi=T_A\circ\varphi_A$, where 
$$\varphi_{A}(s)=\varphi(s)-A \hbox{\quad and}\quad T_{A}(s)=s+A.$$
Since $\varphi_{A}(\C_0)\subset \C_0$, we see from Theorem~\ref{FRED} that $C_{\varphi_A}$ maps ${\Hp}$ into itself. Now the semi-group property and  the ideal property of approximation numbers (see Subsection~\ref{sn}), as well as the previous special case, give
$$a_{n}(C_\varphi)=a_{n}(C_{\varphi _{A}}\circ C_{T_A})\leq \Vert C_{\varphi_A}\Vert a_{n}(C_{T_A})\ll n^{-A}.$$

\subsection*{Part (b), p=1}
 It is easy to conclude from  Lemma \ref{parsum}. Indeed, repeating the proof of the contraction principle for Schauder bases (Lemma~\ref{contract}) and using that 
 \[  \sum_{j\geq n}\frac{\log j}{j^{A+1}}\ll \frac{\log n}{n^{A}},\]
 we obtain 
\[\Big\Vert \sum_{j\geq n}j^{-A}x_j e_j\Big\Vert_1\ll \frac{\log n}{n^{A}}\Big\Vert \sum_{j\geq1} x_j e_j\Big\Vert_1.\] 
 We get $a_{n}(C_{T_A})\ll  (\log n)/n^{A}$ for $T_{A}(s)=s+A$ and conclude as before in the general case $\Real \varphi(s)>A$.  
 \end{proof}

\section{A transference principle}\label{transfer}


In \cite{QS1}, we found a recipe for transferring  a general composition operator on $H^{2}(\D)$ to a composition operator on $\Ht$. The point was that, under this transference, decay rates for approximation numbers are preserved or at least not perturbed severely.  The same transference makes sense in the $\Hp$ setting, but we succeed only partially in getting similarly precise results as in \cite{QS1}.   We will now present this state of affairs and briefly describe the two basic problems that prevent us from proceeding further.   

We begin by describing the recipe from \cite{QS1}. Given $1\le p<\infty$, we let $T$ be some conformal map from $\D$ into $\C_{1/2}$, which we will assume has the property that the operator $C_T$ is bounded from $\Hp$ to $H^p(\D)$. We introduce the function
$$I(s):=2^{-s}$$
which we view as an analytic map from $\C_{0}$ onto $\D\backslash\{0\}$. If $\omega$ is an analytic self-map of $\D$, then we  define an analytic map $\varphi:\C_0\to \C_{1/2}$ by the formula $ \varphi:=T\circ \omega\circ I$, which implies $C_\varphi=C_I\circ C_\omega\circ C_T$.
The Dirichlet series $\varphi$ is then the symbol of a bounded composition operator $C_{\varphi}$ on ${\Hp}$ with $c_0=0$.

A natural choice is to set $T=T_0$, where
\[ T_0(z):=\frac{1}{2}+\frac{1-z}{1+z}, \]
so that $T$ maps $\D$ onto $\C_{1/2}$. Unfortunately, this forces us to require $p$ to be even integer. 
This constraint comes, as in Theorem~\ref{FRED}, from the local embedding 
\begin{equation}\label{lokal}\sup_{a\in \mathbb{R}}\,\int_{a}^{a+1}|f(1/2+it)|^p\,dt\leq C \Vert f\Vert_{{\Hp}}^{p},\end{equation}
which is only known to hold when $p$ is an even integer. This result relies on a well-known inequality in analytic number theory \cite{M}. See \cite{SaSe} and also \cite{OLSA} for a thorough discussion of this inequality and its connections with Carleson measures. Assuming that $f$ is in $\Hp$ and using \eqref{lokal}, we get
that 
\begin{eqnarray*} \| f\circ T_0\|_{H^p(\D)}^p &=& \int_{-\pi}^\pi \vert f\big(1/2+i\tan(t/2)\big)\vert^{p}\, \frac{dt}{2\pi}=\int_{-\infty}^\infty \vert f(1/2+ix)\vert^{p}\frac{dx}{\pi(1+x^2)} \\ & = & \sum_{k\in \Z}\int_{k}^{k+1} \vert f(1/2+ix)\vert^{p}\frac{dx}{\pi(1+x^2)}\ll \sum_{k\in \Z} \frac{1}{k^2 +1}\Vert f\Vert_{{\Hp}}^{p}\ll \Vert f\Vert_{{\Hp}}^{p}.\end{eqnarray*}
It follows that the composition operator defined by the formula $C_{T}$ is a bounded operator from  ${\Hp}$ to $H^p(\D)$. 

For other values of $p$, we may instead choose, for example,
\[ T_\varepsilon(z):=\frac{1}{2}+\left(\frac{1-z}{1+z}\right)^{1-\varepsilon} \]
for some $0<\varepsilon<1$. Using the pointwise estimates $|f( \sigma+it)|\le [\zeta(2\sigma)]^{1/p}\|f\|_{\Hp}$ along with 
$$2\Real T_{\varepsilon}(z)-1\geq (2\sin \pi\,\varepsilon/2)\,\Big\vert \frac{1-z}{1+z}\Big\vert^{1-\varepsilon},$$ 
we may compute in a similar way as above to get that 
\[ \| f\circ T_\varepsilon\|_{H^p(\D)}^p\ll \int_{-\infty}^\infty \| f \|^{p}_{\Hp} \max(1,|x|^{\varepsilon-1}) \frac{dx}{(1+x^2)}\ll  \|f\|_{\Hp}^p.\]

We are prepared to state our first basic estimate for $C_\varphi$.

 \begin{theorem}\label{tsfc} Let $\omega$ be an analytic self-map of $\D$. Assume that $1\le  p <\infty$ and that $C_T$ is bounded from $H^p(\D)$ into $\Hp$, and set $\varphi:=T \circ \omega \circ I$. Then
\[   a_n(\Cp) \le \| C_T \| a_n(C_\omega). \]
In particular, $C_\varphi$ is compact whenever $C_{\omega}$ is compact. \end{theorem}


\begin{proof}
By Theorem~\ref{blift}, the operator $C_I$, defined by setting $C_{I}g(s):=g(I(s))$, is an isometry from $H^{p}(\D)$ into $\Hp$.  
We use the ideal property of approximation numbers and their preservation under left multiplication by isometries to conclude that 
 $$a_{n}(C_\varphi)=a_{n}(C_\omega\circ C_T)\leq \Vert C_T\Vert a_{n}(C_\omega).$$
 Here $\|C_T\|$ is finite by assumption. 
\end{proof}

We would like to have a tight bound on $C_\varphi$ from below as well, but this is harder to achieve. 
We may adapt Theorem  \ref{belowc2bis} to get the following general bound. Here we use the notation $\mu_Z:=\sum_{j=1}^n (1-|z_j|^2)\,\delta_{z_j}$ for a sequence $Z=(z_j)$ in the unit disc.

 \begin{theorem}\label{belowc4} 
Let $\omega$ be an analytic self-map of $\ \D$ such that $\omega(\D)$ has  positive distance to $-1$. Assume that $1\le  p <\infty$ and that $C_T$ is bounded from $H^p(\D)$ into $\Hp$, and set $\varphi:=T \circ \omega \circ I$ and $\ \Phi=T\circ \omega$, the Bohr lift of $\varphi$. There exists a positive constant $c$ such that if $Z=(z_j)$ is any finite sequence with both $Z$ and $\omega(Z)$ consisting of $n$ distinct points in $\D$, then 
\[ a_{n}(\Cp)\geq  c\,n^{-(1/\min(2,p)-1/p)}\, [M_{{\Hp}}(\Phi(Z)) ]^{-1} \|\mu_{Z}\|_{{\mathcal C}, H^p(\D)}^{-1/p} \inf_{1\le j\le n} \left(\frac{1-|z_j|^2}{1-|\omega(z_j)|^2}\right)^{1/p}. \]
\end{theorem}

\begin{proof} Since
 $\Phi$ is bounded and $|1+\omega(z)|\gg1$, we have (e.g. in the case $T=T_0$)
\[ \zeta(2\Real s_j)=\zeta(2 \Real \Phi(z_j))\ge \frac{c}{2\Real \Phi(z_j)-1}= \frac{c}{2}\frac{ |1+\omega(z_j)|^2}{(1-|\omega(z_j)|^2)}\gg (1-|\omega(z_j)|^2)^{-1}. \]
Using this fact, and following the same reasoning as in the proof of \cite[Theorem~9.1]{QS1}, we obtain the result from Theorem~\ref{belowc2bis}.
\end{proof}
 
This result is completely analogous to the bound from below in \cite[Theorem~9.1]{QS1}, but at present only of interest when $p=1$ because of \eqref{remarkable} which says that $M_{\Ho}(S)\le [M_{\Ht}(S)]^2$ for any $\Ht$ interpolating sequence $S$. 

We now have what we need to present our leading example and thus prove Theorem~\ref{example}. 

\begin{proof}[Proof of Theorem~\ref{example}] When $\omega$ is a lens map, it is known from \cite[Proposition 6.3]{ROQULI} that the approximation numbers decay as $e^{-\sqrt{n}}$ to some positive power. By Theorem~\ref{tsfc}, the same upper bound holds for the decay of $a_n(C_\varphi)$ for the transferred operator\footnote{We mention without proof that, for lens maps, the choice $T=T_0$ would work as well.} $C_\varphi=C_{T_{\varepsilon}} \circ C_{\omega}\circ C_I $ on $\Hp$ for all $p\ge 1$. When $p=1$, we can use Theorem~\ref{belowc4} to arrive at the bound from below. Indeed, if we take $z_j=1-\rho^j$ with $0<\rho<1$ and if $\theta$ denotes the parameter of the lens map $\omega$, then simple estimates, using in particular \eqref{remarkable} and Lemma~\ref{secondlem2}, show that
\begin{eqnarray*} \|\mu_{Z}\|_{{\mathcal C}, H^1(\D)} & \ll & 1,\\
M_{\Ho}(S)& \ll &  [M_{\Ht}(S)]^2\ll [M_{H^{2}(\C_{1/2}}(S)]^\alpha\ll e^{b/(1-\rho)}, \\
\inf_{1\leq j\leq n}\left(\frac{1-|z_j|^2}{1-|\omega(z_j)|^2}\right)^{1/p}& \gg & \rho^{n(1-\theta)/p}.\end{eqnarray*}
We now optimize the choice of $\rho$ by taking $\rho=1-1/\sqrt n$, and we get the lower bound in Theorem ~\ref{example} with the help of Theorem \ref{belowc4}. \end{proof}

We note that for general $p\neq 1,2$, we are not able to get any better result from Theorem~\ref{belowc4} than the general lower bound in part (a) of Theorem~\ref{General}. 

We observe the following limitation of our method when $p$ is not an even integer, and thus in particular also in the case $p=1$. When the approximation numbers of $C_\omega$ decay more slowly than they do when $\omega$ is a lens map,  the approximation numbers of $C_{T_{\varepsilon}} \circ C_{\omega}\circ C_I $ will still decay as a power of $e^{-\sqrt{n}}$ because of the map $T_{\varepsilon}$. Substituting $T_{\varepsilon}$ by a map $T$ which enjoys some smoothness at $z=1$, we may remedy this situation to some extent. But it is clear that we are unable to obtain precise results when, for example,  $(1-|\omega|)^{-1}$ is non-integrable on $\T$.

We conclude that two rather fundamental open problems remain obstacles for extending the utility of our transference principle:
\begin{itemize}
\item Is the embedding inequality \eqref{lokal} valid for a continuous range of $p$, for instance all $1\le p<\infty$?
\item What are the bounded interpolating sequences for $\Hp$ when $1<p<\infty$, $p\neq 2$, and how can the constant of interpolation $M_{\Hp}(S)$ be estimated when the sequence $S$ approaches the vertical line $\Real s=1/2$?
\end{itemize} 
These questions await further investigation.

\section*{Acknowledgement}
The authors are grateful to Eero Saksman for kindly allowing them to include in this paper his unpublished vertical convolution formula \eqref{sakiden} and the alternate proof of Lemma~\ref{parsum}.

\end{document}